\documentclass[a4paper]{amsart}
\usepackage[english]{babel}
\usepackage{microtype}
\usepackage{amsmath, amssymb, amsthm, bbm, nicefrac, stackrel}
\usepackage{xspace}
\usepackage{mathrsfs}
\usepackage{varioref}
\usepackage[shortlabels]{enumitem}
\usepackage[arrow, matrix, curve]{xy} 
	\CompileMatrices
\usepackage{url}
\usepackage{csquotes}
\usepackage{color}
\PassOptionsToPackage{final}{graphicx}
\usepackage{tikz}
	\usetikzlibrary{decorations.pathmorphing}
	\usetikzlibrary{arrows,patterns}
	\usetikzlibrary{shapes}
\usepackage{braids}
\usepackage{mdwlist}
\usepackage{xargs}

\usepackage[pdftitle={Regularity for braided multiplicative unitaries},
pdfauthor={David~B\"ucher and Sutanu Roy}, 
pdfsubject={Mathematics; MSC }
]{hyperref}
\usepackage[lite]{amsrefs}
\newcommand*{\MRref}[2]{ \href{http://www.ams.org/mathscinet-getitem?mr=#1}{MR #1}}

\renewcommand{\PrintDOI}[1]{\href{http://dx.doi.org/#1}{DOI #1}%
  \IfEmptyBibField{volume}{, (to appear in print)}{}}

\usepackage{hyperref}
\hypersetup{final}

\newcommand{\scalefactor}{.66}			

\newcommand*{\alg}{\mathcal}
\newcommand*{\hilb}{\mathscr}

\newcommand*{\N}{\mathbb N}

\newcommand*{\C}{\mathbb C}
\newcommand*{\one}{\mathbbm 1}

\renewcommand*{\H}{\hilb H}
\renewcommand{\L}{\hilb L}
\newcommand*{\K}{\hilb K}
\newcommand*{\B}{\alg B}
\newcommand{\chL}{{\check L}}
\newcommand*{\Kp}{\alg K}
\newcommand{\id}{\textnormal{id}}
\newcommand{\skapro}[1]{\langle {#1} \rangle}
\newcommand{\Cat}{\mathcal C}
\newcommand{\YD}{\mathcal{YD}}
\newcommand{\FYDF}{{{}_F\YD^F(\Cat)}}
\newcommand{\WYDW}{{{}_W\YD^W(\Cat)}}
\newcommand{\Hilb}{\textnormal{Hilb}}

\newcommand{\Hom}{\textnormal{Hom}}
\newcommand*{\End}{\textnormal{End}}
\newcommand*{\ot}{\otimes}
\newcommand*{\bt}{\boxtimes}

\DeclareMathOperator{\hbt}{\hat\boxtimes}
\DeclareMathOperator{\habt}{\hat{\overline{\boxtimes}}}
\newcommand{\hatA}{\hat A}

\newcommandx*{\Malg}{M}
\newcommandx*{\U}{\mathcal{U}}
\newcommand{\Calg}{{C^*\textnormal{-Alg}}}
\newcommand{\algA}{\alg A}
\newcommand{\For}{\mathrm{For}}
\newcommand{\braL}{\langle\L|}
\newcommand{\ketL}{|\L\rangle}

\newcommand{\altern}[2]{
\raisebox{.8ex}{#1}
\ensuremath{\mkern-16mu\rotatebox{-17}{\bigg/}\mkern-16mu}
\raisebox{-.8ex}{#2}}
\newcommand{\altops}[2]{
\ \raisebox{.8ex}{\ensuremath{#1}}
\ensuremath{\mkern-6mu}\rotatebox{-15}{\big/}\ensuremath{\mkern-6mu}
\raisebox{-.8ex}{\ensuremath{#2}}\ }
\newcommand{\inceq}{\altops{\subseteq}{=}}
\newcommand{\rinceq}{\altops{\supseteq}{=}}

\newcommand{\bra}[2]{
\node[draw,fill=white,minimum width=.0em,
	inner sep=.2em,
	regular polygon,regular polygon sides=3,
	shape border rotate=0,
	yscale=.8] at (#1,#2) {};
}
\newcommand{\ket}[2]{
\node[draw,fill=white,minimum width=.0em,
	inner sep=.2em,
	regular polygon,regular polygon sides=3,
	shape border rotate=60,
	yscale=.8] at (#1,#2) {};
}

\numberwithin{equation}{section}

\theoremstyle{plain}
	\newtheorem{lem}{Lemma}
	\newtheorem{prop}[lem]{Proposition}
	\newtheorem{thm}[lem]{Theorem}
	
\theoremstyle{definition}
	\newtheorem{df}[lem]{Definition}
\theoremstyle{remark}
	\newtheorem{rmk}[lem]{Remark}

\begin{document}

\title{Regularity for braided multiplicative unitaries}

\author{David~B\"ucher}
\email{dbuecher@uni-math.gwdg.de}
\address{Mathematisches Institut\\
  Georg-August Universit\"at G\"ottingen\\
  Bunsenstra{\ss}e 3--5\\
  37073 G\"ottingen\\
  Germany}

\author{Sutanu Roy}
\email{sutanu@uni-math.gwdg.de}
\address{Indian Statistical Institute\\
 203, B.T. Road, Kolkata 700108 \\
 India}

\begin{abstract}
We generalise the notions of semi-regularity, regularity, 
and bi-regularity to unitary solutions of the braided Pentagon equation 
in concrete $W^*$-categories with 
semi-regular/regular/bi-regular braiding, and study their 
properties. 
We show, for example, that under bi-regularity-assumptions, 
the leg-algebras form braided $C^*$-bialgebras. 
Moreover, we \enquote{close the circle} for the representation categories: 
The braiding on the Yetter-Drinfeld category of 
a \mbox{(semi-)}regular multiplicative unitary in a category 
with (semi-)regular braiding 
is again (semi-)regular.
\end{abstract}

\subjclass[2010]{Primary 46L89, 22D25; Secondary 81R50}
\keywords{Braided multiplicative unitaries, Regularity, Yetter-Drinfeld category}

\maketitle

\section{Introduction}
Multiplicative unitaries, i.e.\@ unitary operators 
$W$ acting on $\K \ot \K$, for some 
Hilbert space $\K$, that satisfy the 
\emph{Pentagon equation} 
\begin{align}
 \label{eq:ordPE}
  W_{23}W_{12} = W_{12}W_{13}W_{23} 
  &\qquad\text{in~$\U(\K\ot \K\ot \K)$} \ ,
\end{align}
provide a fairly algebraic 
approach to locally compact quantum groups. 
Namely, given a multiplicative unitary $W$, one 
wants to construct a quantum group as follows. 
The set of slices $(\id \ot \omega)(W)$, 
where $\omega$ runs over the normal linear 
functionals on $\B(\K)$, forms an algebra $\hatA_0(W)$. 
Passing to the norm-closure produces 
a Banach algebra $\hatA(W)$. 
In order to make sure that $\hatA(W)$ is a $C^*$-algebra
carrying a bialgebra structure, 
one needs to impose further conditions on $W$, 
like \emph{regularity} or \emph{manageability}, 
see \cite{Baaj-Skandalis:Unitaires, Woronowicz:Mult_unit_to_Qgrp}. 
In this paper, we generalise the notion of 
regularity to braided multiplicative unitaries. 

Our setting is the following: 
We fix a concrete monoidal $W^*$-category $\Cat$, 
i.e.\@ a monoidal $W^*$-category together with a 
monoidal fibre functor to the category of 
Hilbert spaces. On $\Cat$, we fix a unitary 
braiding $c = (c_{H,K})_{H,K\in\Cat}$. 
Then, by a \emph{braided multiplicative unitary}, 
or, a \emph{multiplicative unitary $F$ in $\Cat$} 
we understand a unitary solution 
$F \in \U_\Cat(L \ot L)$ 
to the 
Pentagon equation in $\Cat$, also called fusion equation 
in \cite{Street:Fusion_Operator}, 
\begin{align}
F_{23}F_{12} = F_{12} (c^{-1}_{L,L})_{12} F_{23} (c_{L,L})_{12} F_{23} \ .
\end{align}
We define $\hat A(F) \subseteq \B(\L)$ 
as the norm closure of the set of slices 
$(\id \ot \omega)(F)$, 
where $\omega$ runs over the normal linear 
functionals on the algebra of bounded linear 
operators $\B(\L)$ on the image $\L$ of $L$ under 
the fibre functor. 
We stress that we 
wish to let $\omega$ run over all normal linear 
functionals on $\B(\L)$, and not just over 
those on $\End_\Cat(L)$, 
so this is the point why we require $\Cat$ to 
be a \emph{concrete} monoidal $W^*$-category. 
In contrast to the non-braided case, it is no 
more automatic that $\hat A(F)$ is even an 
algebra. 
The way we assure this here is to impose 
a regularity-condition on the braiding $c_{L,L}$
(see Definition \ref{df:regularity-for-braiding} and 
Proposition \ref{prop:F-slices-are-non-deg-algebras}). 

In Section \ref{ssec:to-C*-algebras} we introduce 
the notion of regularity for a braided multiplicative 
unitary $F$ and show, as a first consequence, 
that it implies that $\hatA(F)$ 
is a $C^*$-algebra (Proposition \ref{prop:hat-AF-is-a-C*-alg}). 
We then wish to obtain an analoge of the fact -- 
known in the non-braided setting -- that $\hatA(W)$ carries a $C^*$-bialgebra 
structure. 
To this end, we have to first relax the notion of 
$C^*$-bialgebra: 
An ordinary $C^*$-bialgebra is a coalgebra 
object in the tensor category $\Calg$ with objects being $C^*$-algebras, 
morphisms $f \in \Hom_\Calg(A, B)$ being non-degenerate 
$^*$-homomorphisms from $A$ to the multiplier algebra $\Malg(B)$ 
of $B$, and with the minimal $C^*$-tensor product $\ot_\mathrm{min}$. 
We define a \emph{braided $C^*$-bialgebra} as 
a coalgebra object in a \emph{monoidal category of $C^*$-algebras}, 
where we allow for a tensor product different from $\ot_\mathrm{min}$, 
see Definitions \ref{df:monoidal-cat-of-C*-algebras} \& 
\ref{df:braided-C*-bialgebra}. 
We go on in Section \ref{ssec:category-hat-A(C)} to define 
a suitable monoidal category of $C^*$-algebras $\hatA(\Cat)$, 
and show in Section \ref{ssec:hat-A(F)-as-C*-bialgebra} 
that $\hatA(F)$ is indeed a $C^*$-bialgebra in this category.

The regularity-condition for braidings, which we 
imposed in the beginning, is nicely compatible with 
the regularity-condition of a braided multiplicative 
unitary: 
The braiding on the category of Yetter-Drinfeld modules associated 
with a regular braided multiplicative unitary $F$ 
in a regularly braided category satisfies itself 
the regularity-condition. 
This we show in Section \ref{sec:YD-categories}. 
Moreover, regularity is inherited under taking 
semi-direct products of multiplicative unitaries, 
as we see in Section \ref{sec:semi-direct-products}. 
Finally, in Section \ref{sec:biregularity}, 
we include some comments on bi-regularity.

\section{Preliminaries}

\subsection{General notation regarding Hilbert spaces}
\label{sec:prelim-Hilb-sp-notation}

We denote Hilbert spaces by script letters 
$\H,\K,\L$, etc. 
The set of bounded operators between Hilbert 
spaces $\H,\K$ or on a Hilbert space $\H$ 
is referred to as $\B(\H,\K)$, respectively, $\B(\H)$; 
the respective subsets of compact operators are denoted 
by $\Kp(\H,\K)$ and $\Kp(\H)$; 
the respective pre-duals are denoted by 
$\B(\H,\K)_*$ and $\B(\H)_*$. 
If $X \subseteq \B(\K,\L)$ is a subset, 
then we denote the closed linear span of $X$ by 
$[X]$.

\subsection{Monoidal $W^*$-categories}

We recall here some definitions related to concrete 
monoidal $W^*$-categories. 
See \cites{Ghez-Lima-Roberts:Wstar, Woronowicz:Tannaka-Krein} for references.

A \emph{$C^*$-category} is a category $\Cat$ 
such that 
\begin{itemize*}
\item 
$\Hom_\Cat(X,Y)$ is a Banach space for all 
objects $X,Y \in \Cat$, and 
$\|f \circ g\| \leq \|f\| \|g\|$ for all $f \in \Hom_\Cat(Y,Z), g \in \Hom_\Cat(X,Y)$, 
\item 
there are isometric involutions $^*: \Hom_\Cat(X,Y) \to \Hom_\Cat(Y,X)$ with 
$(f \circ g)^* = g^* \circ f^*$ for all $f \in \Hom_\Cat(Y,Z), g \in \Hom_\Cat(X,Y)$, 
\item $\|f^* f\| = \|f\|^2$ for all $f \in \Hom_\Cat(X,Y)$. 
\end{itemize*}
For an object $X \in \Cat$, we write $\End_\Cat(X)$ 
for $\Hom_\Cat(X,X)$ and $\U_\Cat(X)$ for the set of 
endomorphisms $u \in \End_\Cat(X)$ such 
that $u^*u = \id_X = uu^*$ (\emph{unitary} endomorphisms). 

A \emph{$W^*$-category} is a $C^*$-category $\Cat$ where for all objects 
$X,Y \in \Cat$ there is some Banach space $\Hom_\Cat(X,Y)_*$ whose 
dual space is $\Hom_\Cat(X,Y)$. 
We refer to \cite{Ghez-Lima-Roberts:Wstar} 
for basic definitions and results on $W^*$-categories. 
A functor $F: \Cat \to \Cat'$ between $W^*$-categories is 
called \emph{normal} if the induced maps $\Hom_\Cat(X,Y) \to \Hom_{\Cat'}(F(X),F(Y))$ 
are continuous with respect to the ultraweak topology, 
which is the locally convex topology 
induced by the predual spaces $\Hom_\Cat(X,Y)_*$, $\Hom_{\Cat'}(F(X),F(Y))_*$. 
A monoidal $W^*$-category is a $W^*$-category $\Cat$, together with 
a normal bi-functor $\ot: \Cat \times \Cat \to \Cat$, a distinguished object 
$\one \in \Cat$, and 
\begin{itemize*}
\item 
natural isomorphisms 
$\lambda_X: \one \ot X \to X$, 
$\rho_X: X \ot \one \to X$, 
for all $X \in \Cat$,
\item 
natural isomorphisms 
$\alpha_{X,Y,Z}: (X \ot Y) \ot Z \to X \ot (Y \ot Z)$, 
for all objects $X,Y,Z$ of $\Cat$, 
\end{itemize*}
subject to the usual coherence conditions. 
A monoidal $W^*$-category is called \emph{strict}, 
if all coherence isomorphisms are given by identities. 
While our statements in the sequel of this paper 
do not assume the monoidal categories involved 
to be strict, we will usually leave out the coherence 
isomorphisms from the notation 
(by Mac~Lane's coherence theorem, all admissible 
insertions of coherence isomorphisms give the same result). 

A \emph{concrete} monoidal $W^*$-category is a 
monoidal $W^*$-category $\Cat$ together with a normal strong monoidal functor 
$\For: \Cat \to \Hilb$, where $\Hilb$ is the category 
of Hilbert spaces and bounded linear maps, with the 
Hilbert-space-completed tensor product. 
A \emph{braided $W^*$-category} is a monoidal $W^*$-category 
($\Cat$, $\ot$, $\one$), together with 
a family $c$ of natural isomorphisms $c_{X,Y}:\, X \ot Y \to Y \ot X$ 
such that $c_{U,V \ot W} = (\id_V \ot c_{U,W}) \circ (c_{U,V} \ot \id_W)$ 
and $c_{U \ot V,W} = (c_{U,W} \ot \id_V) \circ (\id_U \ot c_{V,W})$ 
hold for all objects $U,V,W$ of $\Cat$. 
In the following, we will denote the Hilbert spaces 
$\For(H), \For(K), \For(L)$ 
\enquote{underlying} the objects $H,K,L$ of a concrete 
monoidal $W^*$-category $\Cat$ by script letters 
$\H,\K,\L$ -- just as we also denote ordinary Hilbert spaces. 
This is to keep the notation lighter.

\subsection{Graphical notation}

We use a graphical notation for morphisms in braided categories 
similar to the one used in, e.g., 
\cite{Bakalov-Kirillov:Tens_cat}*{Ch.\,2.3}. 
In particular, our pictures are read from bottom to top. 
For example, a (generic) morphism $f:\, X \to Y$ and 
the braiding $c_{X,Y}:\, X \ot Y \to Y \ot X$ are represented by 
\begin{align}
\scalebox{\scalefactor}{
\begin{tikzpicture}[baseline=-\the\dimexpr\fontdimen22\textfont2\relax+.8cm, line width=.1em]
\draw (0,0) -- (0,1.5);
\node at (0,-.25) {$X$};
\node at (0,1.75) {$Y$};
\node[draw,fill=white] at (0,.75) {$f$};
\end{tikzpicture}} 
 	\quad ,
 	\qquad
 	\text{respectively,} 
 	\qquad
\scalebox{\scalefactor}{
\begin{tikzpicture}[baseline=-\the\dimexpr\fontdimen22\textfont2\relax+.6cm, line width=.1em]
\braid[border height=0cm,height=1.25cm,width=1cm] at (0,1.25) s_1^{-1};
\node at (0,-.25) {$X$};
\node at (1,-.25) {$Y$};
\end{tikzpicture}} 
\ .
\end{align}
For the sake of notational compactness, we represent a morphism 
of the form $F \in \End_\Cat(L \ot L)$ by just a 
horizontal, potentially dashed line, instead of a box,
\begin{align}\label{eq:lines-for-morphisms}
\scalebox{\scalefactor}{
\begin{tikzpicture}[baseline=-\the\dimexpr\fontdimen22\textfont2\relax, line width=.1em]
\draw (0,.6) -- (0,-.6);
\draw (1,.6) -- (1,-.6);
\draw (0,0) -- (1,0);
\node at (1.25,0) {$F$};
\node at (0,-.85) {$L$};
\node at (1,-.85) {$L$};
\node at (0,.85) {$L$};
\node at (1,.85) {$L$};
\end{tikzpicture}} 
\ := \, 
\scalebox{\scalefactor}{
\begin{tikzpicture}[baseline=-\the\dimexpr\fontdimen22\textfont2\relax, line width=.1em]
\draw (0,.6) -- (0,-.6);
\draw (1,.6) -- (1,-.6);
\node[draw,fill=white] at (.5,0) {$\hspace{.9em}F\hspace{.9em}$};
\node at (0,-.85) {$L$};
\node at (1,-.85) {$L$};
\node at (0,.85) {$L$};
\node at (1,.85) {$L$};
\end{tikzpicture}} 
\ =: \, 
\scalebox{\scalefactor}{
\begin{tikzpicture}[baseline=-\the\dimexpr\fontdimen22\textfont2\relax, line width=.1em]
\draw (0,.6) -- (0,-.6);
\draw (1,.6) -- (1,-.6);
\draw[dashed] (0,0) -- (1,0);
\node at (1.25,0) {$F$};
\node at (0,-.85) {$L$};
\node at (1,-.85) {$L$};
\node at (0,.85) {$L$};
\node at (1,.85) {$L$};
\end{tikzpicture}} 
\quad .
\end{align}
We will often omit the label \enquote{$F$} at places, where 
we mean a previously fixed solution $F$ to the 
Pentagon equation in $\Cat$; 
in those cases, solid lines represent $F$, 
while dashed lines are reserved for its conjugate $F^*$. 
By naturality of the half-braiding, one has the following 
equality:
\begin{align}
\scalebox{\scalefactor}{
\begin{tikzpicture}[baseline=-\the\dimexpr\fontdimen22\textfont2\relax, line width=.1em, scale=.85]
\braid[border height=0cm,height=1cm,width=1cm] at (0,1) s_1 s_1^{-1};
\draw (2,1) -- (2,-1);
\draw (1,0) -- (2,0);
\end{tikzpicture}} 
\ &= \
\scalebox{\scalefactor}{
\begin{tikzpicture}[baseline=-\the\dimexpr\fontdimen22\textfont2\relax, line width=.1em, scale=.85]
\braid[border height=0cm,height=1cm,width=1cm] at (1,1) s_1^{-1} s_1;
\draw (0,1) -- (0,-1);
\draw (0,0) -- (1,0);
\end{tikzpicture}} \quad .
\end{align}
This we take as a justification to write, more symmetrically,
\begin{align}
\scalebox{\scalefactor}{
\begin{tikzpicture}[baseline=-\the\dimexpr\fontdimen22\textfont2\relax, line width=.1em, scale=.85]
\braid[border height=0cm,height=1cm,width=1cm] at (0,1) s_1 s_1^{-1};
\draw (2,1) -- (2,-1);
\draw (1,0) -- (2,0);
\end{tikzpicture}} 
\ &=: \
\scalebox{\scalefactor}{
\begin{tikzpicture}[baseline=-\the\dimexpr\fontdimen22\textfont2\relax, line width=.1em, scale=.85]
\draw (0,1) -- (0,-1);
\draw (1,1) -- (1,-1);
\draw (2,1) -- (2,-1);
\draw[color=white,line width=.7em] (.5,0) -- (1.5,0);
\draw (0,0) -- (2,0);
\end{tikzpicture}} \quad .
\end{align}
If $\Cat$ is a \emph{concrete} braided $W^*$-category 
and $H\in\Cat$ is an object with underlying Hilbert space $\H$, 
then we allow for diagrams involving triangles 
\scalebox{\scalefactor}{
\begin{tikzpicture}[baseline=-\the\dimexpr\fontdimen22\textfont2\relax+.125cm, line width=.1em]
\draw (0,0) -- (0,.5);
\ket{0}{0};
\node at (-.25,.35) {$H$};
\end{tikzpicture}} 
and 
\scalebox{\scalefactor}{
\begin{tikzpicture}[baseline=-\the\dimexpr\fontdimen22\textfont2\relax+.25cm, line width=.1em]
\draw (0,0) -- (0,.5);
\bra{0}{.5};
\node at (-.25,.15) {$H$};
\end{tikzpicture}}, 
which we use to denote the sets 
$\B(\C,\H) \cong \H$ 
and 
$\B(\H,\C)$, respectively. 
Note that diagrams involving such triangles 
do in general not represent (sets of) morphisms in $\Cat$, 
but just sets of bounded operators between Hilbert spaces 
underlying objects of $\Cat$. 
With these conventions we have, for example, 
\begin{align}
\left[ 
\scalebox{\scalefactor}{
\begin{tikzpicture}[baseline=-\the\dimexpr\fontdimen22\textfont2\relax, line width=.1em]
\draw (0,.25) -- (0,1);
\draw (0,-.25) -- (0,-1);
\node[draw,fill=white,minimum width=.0em,
	inner sep=.2em,
	regular polygon,regular polygon sides=3,
	shape border rotate=60,
	yscale=.8,label=center:] at (0,.25) {};
\node[draw,fill=white,minimum width=.0em,
	inner sep=.2em,
	regular polygon,regular polygon sides=3,
	shape border rotate=0,
	yscale=.8,label=center:] at (0,-.25) {};
\node at (.25,.83) {$K$};
\node at (.25,-.83) {$H$};
\end{tikzpicture}} 
\right] 
= \Kp(\H,\K) \ ,
\end{align}
if $H,K \in \Cat$ with underlying Hilbert spaces 
$\H$ and $\K$.
Here, the brackets are the notation introduced in 
Section \ref{sec:prelim-Hilb-sp-notation} 
for the closed linear span. 

We will repeatedly make use of the elementary fact that, if 
$X \in \U_\Cat(H_1 \ot K_1, H_2 \ot K_2)$ is a 
unitary, then the set 
$\{X(\eta \ot \xi)\,|\, \eta \in \H_1,\, \xi \in \K_1\}$ 
is dense in $\H_2 \ot \K_2$. This reads in 
our graphical notation 
\begin{align}\label{eq:remove-unitaries}
\left[ 
\scalebox{\scalefactor}{
\begin{tikzpicture}[baseline=-\the\dimexpr\fontdimen22\textfont2\relax-.1cm, line width=.1em]
\draw (0,-.5) -- (0,.5);
\draw (1,-.5) -- (1,.5);
\draw (0,0) -- (1,0);
\ket{0}{-.5};
\ket{1}{-.5};
\node[anchor=east,inner sep=.2em] at (0,.35) {$H_2$};
\node[anchor=west,inner sep=.2em] at (1,.35) {$K_2$};
\node[anchor=south,inner sep=.2em] at (.5,0) {$X$};
\node[anchor=east,inner sep=.5em] at (0,-.5) {$H_1$};
\node[anchor=west,inner sep=.5em] at (1,-.5) {$K_1$};
\end{tikzpicture}} 
\right] 
= 
\left[ 
\scalebox{\scalefactor}{
\begin{tikzpicture}[baseline=-\the\dimexpr\fontdimen22\textfont2\relax-.1cm, line width=.1em]
\draw (0,-.5) -- (0,.5);
\draw (1,-.5) -- (1,.5);
\ket{0}{-.5};
\ket{1}{-.5};
\node[anchor=east,inner sep=.2em] at (0,.35) {$H_1$};
\node[anchor=west,inner sep=.2em] at (1,.35) {$K_1$};
\end{tikzpicture}} 
\right]  \ .
\end{align}

\subsection{Braided multiplicative unitaries}
\label{ssec:br-mult-unitaries}

In the following we fix a strict monoidal unitarily braided 
$W^*$-category $(\Cat,\ot,\one,c)$. 

\medskip

A morphism $F \in \U_{\Cat}(L \ot L)$ is called 
\emph{multiplicative unitary \textnormal{(}in $\Cat$}), if 
it satisfies the \emph{Pentagon equation} 
\begin{align}\label{eq:br-pentagon-1}
\scalebox{\scalefactor}{
\begin{tikzpicture}[baseline=-\the\dimexpr\fontdimen22\textfont2\relax, line width=.1em]
\draw (0,.75) -- (0,-.75);
\draw (1,.75) -- (1,-.75);
\draw (2,.75) -- (2,-.75);
\draw (0,-.3) -- (1,-.3);
\draw (1,.3) -- (2,.3);
\node at (-.25,-.3) {$F$};
\node at (2.25,.3) {$F$};
\node at (0,-1) {$L$};
\node at (1,-1) {$L$};
\node at (2,-1) {$L$};
\end{tikzpicture}} 
\ \, = \
\scalebox{\scalefactor}{
\begin{tikzpicture}[baseline=-\the\dimexpr\fontdimen22\textfont2\relax, line width=.1em]
\draw (0,.75) -- (0,-.75);
\draw (1,.75) -- (1,-.75);
\draw (2,.75) -- (2,-.75);
\draw (0,.5) -- (1,.5);
\draw (1,-.5) -- (2,-.5);
\draw[color=white,line width=.7em] (.5,0) -- (1.5,0);
\draw (0,0) -- (2,0);
\node at (-.25,.5) {$F$};
\node at (2.25,0) {$F$};
\node at (2.25,-.5) {$F$};
\node at (0,-1) {$L$};
\node at (1,-1) {$L$};
\node at (2,-1) {$L$};
\end{tikzpicture}} 
\quad .
\end{align}
More explicitly this means 
$F_{23} F_{12} = F_{12} (c_{L,L})_{12} F_{23} (c^{-1}_{L,L})_{12} F_{23}$, 
using the \enquote{leg notation} 
$F_{12} := F \ot \id_L$ and 
$F_{23} := \id_L \ot F$.

The category of (right-)corepresentations of a multiplicative 
unitary $F \in \U_{\Cat}(L \ot L)$, denoted 
by $\Cat^F$, is the monoidal category given by 
\begin{itemize*}
\item objects: 
pairs $(H,U)$, where $H \in \Cat$ is an object and 
$U \in \U_\Cat(H \ot L)$ satisfies 
\begin{align}\label{eq:br-pentagon-2}
\scalebox{\scalefactor}{
\begin{tikzpicture}[baseline=-\the\dimexpr\fontdimen22\textfont2\relax, line width=.1em]
\draw (0,.75) -- (0,-.75);
\draw (1,.75) -- (1,-.75);
\draw (2,.75) -- (2,-.75);
\draw (0,-.3) -- (1,-.3);
\draw (1,.3) -- (2,.3);
\node at (-.25,-.3) {$U$};
\node at (2.25,.3) {$F$};
\node at (0,-1) {$H$};
\node at (1,-1) {$L$};
\node at (2,-1) {$L$};
\end{tikzpicture}} 
\ \, = \
\scalebox{\scalefactor}{
\begin{tikzpicture}[baseline=-\the\dimexpr\fontdimen22\textfont2\relax, line width=.1em]
\draw (0,.75) -- (0,-.75);
\draw (1,.75) -- (1,-.75);
\draw (2,.75) -- (2,-.75);
\draw (0,.5) -- (1,.5);
\draw (1,-.5) -- (2,-.5);
\draw[color=white,line width=.7em] (.5,0) -- (1.5,0);
\draw (0,0) -- (2,0);
\node at (-.25,.5) {$U$};
\node at (2.25,0) {$U$};
\node at (2.25,-.5) {$F$};
\node at (0,-1) {$H$};
\node at (1,-1) {$L$};
\node at (2,-1) {$L$};
\end{tikzpicture}} 
\quad .
\end{align}
\item morphisms $(H_1,U_1) \to (H_2,U_2)$: 
morphisms $f \in \Hom_\Cat(H_1,H_2)$ such that 
$(f \ot \id_L) \circ U_1 = U_2 \circ (f \ot \id_L)$. 
\item tensor product: 
for $(H_1,U_1), (H_2,U_2) \in \Cat^F$, 
\begin{align}
(H_1,U_1) \ot (H_2,U_2) &:= \biggl( H_1 \ot H_2, 
\scalebox{\scalefactor}{
\begin{tikzpicture}[baseline=-\the\dimexpr\fontdimen22\textfont2\relax-.4cm, line width=.1em]
\draw (0,.3) -- (0,-.75);
\draw (1,.3) -- (1,-.75);
\draw (2,.3) -- (2,-.75);
\draw (1,-.5) -- (2,-.5);
\draw[color=white,line width=.7em] (.5,0) -- (1.5,0);
\draw (0,0) -- (2,0);
\node at (2.3,0) {$U_1$};
\node at (2.3,-.5) {$U_2$};
\node at (0,-1) {$H_1$};
\node at (1,-1) {$H_2$};
\node at (2,-1) {$L$};
\end{tikzpicture}} 
\biggr) \ . 
\end{align}
\end{itemize*}
Similarly, one defines the category $_F\Cat$ of (left-)representations of $F$: 
\begin{itemize*}
\item objects: pairs $(H,U)$, where $H \in \Cat$ is an object and 
$U \in \U_\Cat(L \ot H)$ satisfies 
\begin{align}
\scalebox{\scalefactor}{
\begin{tikzpicture}[baseline=-\the\dimexpr\fontdimen22\textfont2\relax, line width=.1em]
\draw (0,.75) -- (0,-.75);
\draw (1,.75) -- (1,-.75);
\draw (2,.75) -- (2,-.75);
\draw (0,-.3) -- (1,-.3);
\draw (1,.3) -- (2,.3);
\node at (-.25,-.3) {$F$};
\node at (2.25,.3) {$U$};
\node at (0,-1) {$L$};
\node at (1,-1) {$L$};
\node at (2,-1) {$H$};
\end{tikzpicture}} 
\ \, = \
\scalebox{\scalefactor}{
\begin{tikzpicture}[baseline=-\the\dimexpr\fontdimen22\textfont2\relax, line width=.1em]
\draw (0,.75) -- (0,-.75);
\draw (1,.75) -- (1,-.75);
\draw (2,.75) -- (2,-.75);
\draw (0,.5) -- (1,.5);
\draw (1,-.5) -- (2,-.5);
\draw[color=white,line width=.7em] (.5,0) -- (1.5,0);
\draw (0,0) -- (2,0);
\node at (-.25,.5) {$F$};
\node at (2.25,0) {$U$};
\node at (2.25,-.5) {$U$};
\node at (0,-1) {$L$};
\node at (1,-1) {$L$};
\node at (2,-1) {$H$};
\end{tikzpicture}} 
\quad .
\end{align}
\item morphisms $(H_1,U_1) \to (H_2,U_2)$: 
morphisms $f \in \Hom_\Cat(H_1,H_2)$ such that 
$(\id_L \ot f) \circ U_1 = U_2 \circ (\id_L \ot f)$. 
\item tensor product: 
for $(H_1,U_1), (H_2,U_2) \in {}_F\Cat$, 
\begin{align}
(H_1,U_1) \ot (H_2,U_2) &:= \biggl( H_1 \ot H_2, 
\scalebox{\scalefactor}{
\begin{tikzpicture}[baseline=-\the\dimexpr\fontdimen22\textfont2\relax-.4cm, line width=.1em]
\draw (0,.3) -- (0,-.75);
\draw (1,.3) -- (1,-.75);
\draw (2,.3) -- (2,-.75);
\draw (0,.05) -- (1,.05);
\draw[color=white,line width=.7em] (.5,-.45) -- (1.5,-.45);
\draw (0,-.45) -- (2,-.45);
\node at (-.3,0) {$U_1$};
\node at (-.3,-.5) {$U_2$};
\node at (0,-1) {$L$};
\node at (1,-1) {$H_1$};
\node at (2,-1) {$H_2$};
\end{tikzpicture}} 
\biggr) \ . 
\end{align}
\end{itemize*}
The category $\FYDF$ of Yetter-Drinfeld modules 
of $F$ is given by
\begin{itemize*}
\item objects: triples $(H,U,V)$, where $(H,U) \in \Cat^F$ 
and $(H,V) \in {}_F\Cat$ satisfy 
\begin{align}
\scalebox{\scalefactor}{
\begin{tikzpicture}[baseline=-\the\dimexpr\fontdimen22\textfont2\relax, line width=.1em]
\draw (0,.75) -- (0,-.75);
\draw (1,.75) -- (1,-.75);
\draw (2,.75) -- (2,-.75);
\draw (0,.5) -- (1,.5);
\draw (1,-.5) -- (2,-.5);
\draw[color=white,line width=.7em] (.5,0) -- (1.5,0);
\draw (0,0) -- (2,0);
\node at (-.25,.5) {$V$};
\node at (2.25,0) {$F$};
\node at (2.25,-.5) {$U$};
\node at (0,-1) {$L$};
\node at (1,-1) {$H$};
\node at (2,-1) {$L$};
\end{tikzpicture}} 
\ \, = \
\scalebox{\scalefactor}{
\begin{tikzpicture}[baseline=-\the\dimexpr\fontdimen22\textfont2\relax, line width=.1em]
\draw (1,.5) -- (2,.5);
\draw (0,-.5) -- (1,-.5);
\draw (0,0) -- (2,0);
\draw (0,.75) -- (0,-.75);
\draw[color=white,line width=.7em] (1,.25) -- (1,-.25);
\draw (1,.75) -- (1,-.75);
\draw (2,.75) -- (2,-.75);
\node at (-.25,-.5) {$V$};
\node at (2.25,0) {$F$};
\node at (2.25,.5) {$U$};
\node at (0,-1) {$L$};
\node at (1,-1) {$H$};
\node at (2,-1) {$L$};
\end{tikzpicture}} 
\quad .
\end{align}
\item morphisms $(H_1,U_1,V_1) \to (H_2,U_2,V_2)$: 
morphisms $f \in \Hom_\Cat(H_1,H_2)$ such that 
$(f \ot \id_L) \circ U_1 = U_2 \circ (f \ot \id_L)$ 
and 
$(\id_L \ot f) \circ V_1 = V_2 \circ (\id_L \ot f)$ . 
\item tensor product: 
for $(H_1,U_1,V_1), (H_2,U_2,V_2) \in \FYDF$, 
\begin{align}
(H_1,U_1,V_1) \ot (H_2,U_2,V_2) &:= \biggl( H_1 \ot H_2, 
\scalebox{\scalefactor}{
\begin{tikzpicture}[baseline=-\the\dimexpr\fontdimen22\textfont2\relax-.4cm, line width=.1em]
\draw (0,.3) -- (0,-.75);
\draw (1,.3) -- (1,-.75);
\draw (2,.3) -- (2,-.75);
\draw (1,-.5) -- (2,-.5);
\draw[color=white,line width=.7em] (.5,0) -- (1.5,0);
\draw (0,0) -- (2,0);
\node at (2.3,0) {$U_1$};
\node at (2.3,-.5) {$U_2$};
\node at (0,-1) {$H_1$};
\node at (1,-1) {$H_2$};
\node at (2,-1) {$L$};
\end{tikzpicture}} , 
\scalebox{\scalefactor}{
\begin{tikzpicture}[baseline=-\the\dimexpr\fontdimen22\textfont2\relax-.4cm, line width=.1em]
\draw (0,.05) -- (2,.05);
\draw (0,-.45) -- (1,-.45);
\draw (0,.3) -- (0,-.75);
\draw[color=white,line width=.7em] (1,.3) -- (1,-.25);
\draw (1,.3) -- (1,-.75);
\draw (2,.3) -- (2,-.75);
\node[anchor=east,inner sep=.2em] at (0,0) {$V_2$};
\node[anchor=east,inner sep=.2em] at (0,-.5) {$V_1$};
\node at (0,-1) {$L$};
\node at (1,-1) {$H_1$};
\node at (2,-1) {$H_2$};
\end{tikzpicture}} 
\biggr) \ . 
\end{align}
\end{itemize*}
If $\Cat$ is a concrete braided $W^*$-category, then 
we call a multiplicative unitary $F \in \U_\Cat(L \ot L)$ 
\emph{good}, if for all $a \in \B(\L)$ 
-- where $\L$ is the Hilbert space underlying $L \in \Cat$ -- 
the implication 
\begin{align} \label{eqn:weaker-goodness-condition}
\left( 
\scalebox{\scalefactor}{
\begin{tikzpicture}[baseline=-\the\dimexpr\fontdimen22\textfont2\relax-.2cm, line width=.1em]
\draw (0,1) -- (0,-1);
\draw (1,1) -- (1,-1);
\draw (0,.6) -- (1,.6);
\draw[dashed] (0,-.6) -- (1,-.6);
\node[draw,fill=white] at (0,0) {$a$};
\node at (1.2,.6) {$F$};
\node at (1.3,-.6) {$F^*$};
\node at (0,-1.25) {$L$};
\node at (1,-1.25) {$L$};
\end{tikzpicture}} 
= 
\scalebox{\scalefactor}{
\begin{tikzpicture}[baseline=-\the\dimexpr\fontdimen22\textfont2\relax-.2cm, line width=.1em]
\braid[border height=.0cm,height=1cm,width=1cm] at (0,1) s_1^{-1} s_1;
\node[draw,fill=white] at (0.1,0) {$a$};
\node at (0,-1.25) {$L$};
\node at (1,-1.25) {$L$};
\end{tikzpicture}} 
\right) 
\quad \Longrightarrow \quad 
a \in \C \cdot \id_L \ 
\end{align}
holds true. 
In \eqref{eqn:weaker-goodness-condition} we have, for the first time, 
applied our convention on the usage 
of dashed lines explained below \eqref{eq:lines-for-morphisms}. 
If $F$ is good, then for all 
$(H,U) \in \Cat^F$ and $(K,V) \in {}_F\Cat$ there exists a unique unitary 
$V \ast U \in \U_\Cat(H \ot K)$ such that 
\begin{align}\label{eq:in-lem---good--indep-of-2nd-leg}
\scalebox{\scalefactor}{
\begin{tikzpicture}[baseline=-\the\dimexpr\fontdimen22\textfont2\relax, line width=.1em]
\draw (0,1) -- (0,-1);
\draw (1,1) -- (1,-1);
\draw (2,1) -- (2,-1);
\draw[color=white,line width=.7em] (.5,0) -- (1.5,0);
\draw (0,0) -- (2,0);
\node at (-.55,0) {$V \ast U$};
\node at (0,-1.25) {$H$};
\node at (1,-1.25) {$L$};
\node at (2,-1.25) {$K$};
\end{tikzpicture}} 
\ = 
\scalebox{\scalefactor}{
\begin{tikzpicture}[baseline=-\the\dimexpr\fontdimen22\textfont2\relax, line width=.1em]
\draw (0,1) -- (0,-1);
\draw (1,1) -- (1,-1);
\draw (2,1) -- (2,-1);
\draw[dashed] (1,-.6) -- (2,-.6);
\draw (0,-.2) -- (1,-.2);
\draw (1,.2) -- (2,.2);
\draw[dashed] (0,.6) -- (1,.6);
\node at (-.3,.6) {$U^*$};
\node at (-.25,-.2) {$U$};
\node at (2.25,.2) {$V$};
\node at (2.3,-.6) {$V^*$};
\node at (0,-1.25) {$H$};
\node at (1,-1.25) {$L$};
\node at (2,-1.25) {$K$};
\end{tikzpicture}}  \ ,
\end{align}
cf.\@ the proof of \cite{Meyer-Roy-Woronowicz:Homomorphisms}*{Lem.\,3.11} for the special case that 
$(\Cat,\ot,c) = (\Hilb,\ot,\Sigma)$; 
the general case will be presented in a separate paper. 
There, it will also be shown:
If $F$ is a good multiplicative unitary in $\Cat$, 
then $\FYDF$ is braided with braiding $\Phi$ given by 
$\Phi_{H,K} := c_{K,H}^{-1}\circ(V_K \ast U_H)$, 
for objects $(H,U_H,V_H), (K,U_K,V_K) \in \FYDF$.

\subsection{Monoidal categories of $C^*$-algebras and braided $C^*$-bialgebras}

Given a $C^*$-algebra $A$, we denote its multiplier 
algebra by $\Malg(A)$. 
We denote by $\Calg$ the category whose objects are 
$C^*$-algebras and whose morphisms $f \in \Hom_\Calg(A, B)$ are 
non-degenerate $^*$-homomorphisms $f: A \to \Malg(B)$. 
By non-degeneracy, such a $^*$-homomorphism extends uniquely 
to a $^*$-homomorphism $\Malg(A) \to \Malg(B)$. Hence, 
morphisms in $\Calg$ are indeed composable. 
We refer to, e.g., \cite{Woronowicz:Unbounded_affiliated} for further details.

Let $A,B,C$ be $C^*$-algebras and $\alpha \in \Hom_\Calg(A,C)$ 
and $\beta \in \Hom_\Calg(B,C)$ be morphisms. 
We say that $C$ is a 
\emph{crossed product of $A,B$ with inclusions $\alpha,\beta$}, if 
$C = [\alpha(A) \cdot \beta(B)] \subseteq \Malg(C)$, 
cf.\@ \cite{Meyer-Roy-Woronowicz:Twisted_tensor}*{Def.\,2.11}.

\begin{df}\label{df:monoidal-cat-of-C*-algebras}
A \emph{monoidal category of $C^*$-algebras} is a 
monoidal category $(\Cat,\bt,\one)$ together with 
a faithful functor $\algA: \Cat \to \Calg$ and 
natural transformations 
$\iota_1 : \Pr_1 \to \bt$ and 
$\iota_2 : \Pr_2 \to \bt$, 
where $\Pr_1,\Pr_2$ denote the functors 
$\Cat \times \Cat \to \Cat$ projecting on the 
first, respectively, second component, 
subject to:
\begin{enumerate}
\item 
for all $A,B \in \Cat$ 
we have that $\algA(A \bt B)$ forms a 
crossed product of $\algA(A)$ and $\algA(B)$ 
with inclusions $\algA(\iota_{1;A,B})$ and 
$\algA(\iota_{2;A,B})$. 

\item 
for all $A,B,C \in \Cat$ we have 
\begin{gather*}
\iota_{1;A \bt B,C} \circ \iota_{1;A,B} = \iota_{1;A,B \bt C} \ , \qquad
\iota_{1;A \bt B,C} \circ \iota_{2;A,B} = \iota_{2;A,B \bt C} \circ \iota_{1;B,C} \ , \\
\iota_{2;A \bt B,C} = \iota_{2;A,B \bt C} \circ \iota_{2;B,C} \ .
\end{gather*}
\end{enumerate}
\end{df}

An example of a monoidal category of $C^*$-algebras 
is $\Calg$ itself ($\algA$ being the identity), 
with the minimal $C^*$-tensor product 
$\ot_\mathrm{min}$ and the natural transformations 
$\iota_{1;A,B}: A \to A \ot_\mathrm{min} B$ and 
$\iota_{2;A,B}: B \to A \ot_\mathrm{min} B$ 
given by $\iota_{1;A,B}(a) := a \ot_\mathrm{min} 1_B$ and 
$\iota_{2;A,B}(b) := 1_A \ot_\mathrm{min} b$, 
where $1_B$ and $1_A$ are the units of the 
multiplier algebras $\Malg(A)$, respectively, $\Malg(B)$.

In the following, we will usually suppress the functor $\algA$, 
thinking of a monoidal category of $C^*$-algebras as a 
(non-full) subcategory of $\Calg$, equipped with 
a tensor product. 
Instead of $\algA(\iota_{1;A,B})(a)$ and 
$\algA(\iota_{2;A,B})(b)$, for $a \in A$ and $b \in B$, 
we also write $a \bt 1_B$, respectively, $1_A \bt b$. 

\begin{df}\label{df:braided-C*-bialgebra}
Let $(\Cat,\bt,\one)$ be a monoidal category of $C^*$-algebras. 
A (\emph{braided\textnormal{)} $C^*$-bialgebra} in $\Cat$ is a coalgebra 
$(A,\Delta)$ in $\Cat$, i.e.\@ $A \in \Cat$ and 
$\Delta \in \Hom_\Cat(A,A \bt A)$ such that 
$(\Delta \bt \id)\circ \Delta = (\id \bt \Delta)\circ \Delta$. 
A $C^*$-bialgebra $(A,\Delta)$ in $\Cat$ is called 
\emph{bi-simplifiable} if the \emph{Podle{\'s} conditions} 
$A \bt A = [\Delta(A) \cdot (A \bt 1_A)]$ and 
$A \bt A = [\Delta(A) \cdot (1_A \bt A)]$ hold.
\end{df}

\section{From braided multiplicative unitaries to braided $C^*$-bialgebras}

\subsection{From semi-regularity to $C^*$-algebras}
\label{ssec:to-C*-algebras}

\begin{df}\label{df:regularity-for-braiding}
A braiding $c$ on a concrete monoidal 
$W^*$-category $(\Cat,\ot,\one)$ is said to be 
\emph{semi-regular}, if it is unitary and 
\begin{align}\label{eq:semi-regularity-for-braiding}
\Kp(\H,\K) 
= \left[ 
\scalebox{\scalefactor}{
\begin{tikzpicture}[baseline=-\the\dimexpr\fontdimen22\textfont2\relax, line width=.1em]
\draw (0,.25) -- (0,1);
\draw (0,-.25) -- (0,-1);
\node[draw,fill=white,minimum width=.0em,
	inner sep=.2em,
	regular polygon,regular polygon sides=3,
	shape border rotate=60,
	yscale=.8,label=center:] at (0,.25) {};
\node[draw,fill=white,minimum width=.0em,
	inner sep=.2em,
	regular polygon,regular polygon sides=3,
	shape border rotate=0,
	yscale=.8,label=center:] at (0,-.25) {};
\node at (.25,.83) {$K$};
\node at (.25,-.83) {$H$};
\end{tikzpicture}} 
\right] 
\subseteq 
\left[ 
\scalebox{\scalefactor}{
\begin{tikzpicture}[baseline=-\the\dimexpr\fontdimen22\textfont2\relax, line width=.1em]
\draw (0,.25) -- (0,1);
\draw (0,-.25) -- (0,-1);
\braid[height=.5cm,width=.5cm] at (0,.5cm) s_1^{-1};
\node[draw,fill=white,minimum width=.0em,
	inner sep=.2em,
	regular polygon,regular polygon sides=3,
	shape border rotate=0,
	yscale=.8,label=center:] at (.5,.4) {};
\node[draw,fill=white,minimum width=.0em,
	inner sep=.2em,
	regular polygon,regular polygon sides=3,
	shape border rotate=60,
	yscale=.8,label=center:] at (.5,-.4) {};
\node at (.25,.83) {$K$};
\node at (.25,-.83) {$H$};
\end{tikzpicture}} 
\right] 
= [(\id \ot \omega)(c_{H,K})\,|\, \omega \in \B(\K,\H)_*]
\end{align}
holds for all $H,K \in \Cat$. 
It is said to be \emph{regular}, if one has 
an equality in 
\eqref{eq:semi-regularity-for-braiding}.
\end{df}

\begin{rmk}\phantomsection\label{rmk:regular-braiding}
\begin{enumerate}
\item 
The category $\Hilb$ is regularly braided by 
the tensor flip $\Sigma$. We will see in 
Proposition \ref{prop:YD-cats-of-good-regular-MUs-have-regular-braiding} 
below that more examples of regularly braided categories 
are provided by Yetter-Drinfeld categories of 
regular multiplicative unitaries. 
\item 
If $\Cat$ has a semi-regular unitary braiding $c$, then, 
by taking adjoints, we see that 
\begin{align}
\Kp(\H,\K) 
\subseteq \left[ 
\scalebox{\scalefactor}{
\begin{tikzpicture}[baseline=-\the\dimexpr\fontdimen22\textfont2\relax, line width=.1em]
\draw (0,.25) -- (0,1);
\draw (0,-.25) -- (0,-1);
\braid[height=.5cm,width=.5cm] at (0,.5cm) s_1;
\node[draw,fill=white,minimum width=.0em,
	inner sep=.2em,
	regular polygon,regular polygon sides=3,
	shape border rotate=0,
	yscale=.8,label=center:] at (.5,.4) {};
\node[draw,fill=white,minimum width=.0em,
	inner sep=.2em,
	regular polygon,regular polygon sides=3,
	shape border rotate=60,
	yscale=.8,label=center:] at (.5,-.4) {};
\node at (.25,.83) {$K$};
\node at (.25,-.83) {$H$};
\end{tikzpicture}} 
\right] \ .
\end{align}
It follows that $\Cat^\mathrm{rev}$, which is the same 
as $\Cat$ as a monoidal category, but braided by the 
inverse $(c_{K,H}^{-1})_{H,K\in\Cat}$ of $c$, is 
semi-regularly braided, too. 
An analogous statement holds for regularity instead of 
semi-regularity.
\end{enumerate}
\end{rmk}

\begin{prop}\label{prop:F-slices-are-non-deg-algebras}
Let $\Cat$ be a semi-regularly braided concrete monoidal $W^*$-category, 
and let $L \in \Cat$ be an object with underlying 
Hilbert space $\L$. 
Let $F \in \U_{\Cat}(L \ot L)$ be a multiplicative 
unitary. 
Define 
\begin{align}
\hatA(F) := [(\id \ot \omega)(F)\,|\, \omega \in \B(\L)_*] \ . 
\end{align}
Then: 
\begin{enumerate}
\item $\hatA(F) \subseteq \B(\L)$ is a subalgebra.
\item $[\hatA(F)\L] = \L$, i.e.\@ $\hatA(F)$ acts non-degenerately. 
\end{enumerate}
\end{prop}
\begin{proof}
\emph{Part 1:} We have
\begin{align}\label{eq:hat-A-is-an-algebra}
[\hatA(F) \hatA(F)] 
&= 
\left[ 
\scalebox{\scalefactor}{
\begin{tikzpicture}[baseline=-\the\dimexpr\fontdimen22\textfont2\relax+1.2cm, line width=.1em]
\draw (0,-.25) -- (0,2.75);
\draw (1,0) -- (1,1);
\draw (1,1.5) -- (1,2.5);
\draw (0,.5) -- (1,.5);
\draw (0,2) -- (1,2);
\bra{1}{2.5};
\ket{1}{1.5};
\bra{1}{1};
\ket{1}{0};
\node at (.25,-.05) {$L$};
\node at (.25,2.55) {$L$};
\node at (-.25,.5) {$F$};
\node at (-.25,2) {$F$};
\end{tikzpicture}} 
\right]
\stackrel{\eqref{eq:semi-regularity-for-braiding}}\subseteq 
\left[ 
\scalebox{\scalefactor}{
\begin{tikzpicture}[baseline=-\the\dimexpr\fontdimen22\textfont2\relax+1.2cm, line width=.1em]
\draw (0,-.25) -- (0,2.75);
\draw (1,0) -- (1,1);
\draw (1,1.5) -- (1,2.5);
\draw (0,.5) -- (1,.5);
\draw (0,2) -- (1,2);
\braid[border height=.15cm,height=.5cm,width=.5cm] at (1,1.65) s_1^{-1};
\bra{1}{2.5};
\bra{1.5}{1.65};
\ket{1.5}{.85};
\ket{1}{0};
\end{tikzpicture}} 
\right]
= 
\left[ 
\scalebox{\scalefactor}{
\begin{tikzpicture}[baseline=-\the\dimexpr\fontdimen22\textfont2\relax+1.1cm, line width=.1em]
\draw (0,-.25) -- (0,2.5);
\draw (1,1) -- (1,2.25);
\draw (2,1) -- (2,1.75);
\draw[color=white, line width=.7em] (.5,1.25) -- (1.5,1.25);
\draw (0,1.25) -- (2,1.25);
\draw (0,1.75) -- (1,1.75);
\braid[border height=.0cm,height=1cm,width=1cm] at (1,1) s_1^{-1};
\bra{1}{2.25};
\bra{2}{1.75};
\ket{1}{0};
\ket{2}{0};
\end{tikzpicture}} 
\right]
\stackrel{\eqref{eq:br-pentagon-1}}= 
\left[ 
\scalebox{\scalefactor}{
\begin{tikzpicture}[baseline=-\the\dimexpr\fontdimen22\textfont2\relax+1.1cm, line width=.1em]
\draw (0,-.25) -- (0,2.5);
\draw (1,1) -- (1,2);
\draw (2,1) -- (2,2);
\draw[dashed] (1,1) -- (2,1);
\draw (0,1.25) -- (1,1.25);
\draw (1,1.5) -- (2,1.5);
\braid[border height=.0cm,height=1cm,width=1cm] at (1,1) s_1^{-1};
\bra{1}{2};
\bra{2}{2};
\ket{1}{0};
\ket{2}{0};
\end{tikzpicture}} 
\right]
\nonumber\\
&\stackrel{\eqref{eq:remove-unitaries}}= 
\hatA(F) \ .
\end{align}
Let us make the above graphical calculation more explicit at this point, 
since this is the first time we use our notation in a proof. 
Also, we use the opportunity to spell out the coherence 
isomorphisms. 
Denote by $\alpha^\Cat_{X,Y,Z}$, 
$\alpha^\Hilb_{X,Y,Z}$ the associators, 
and by $\lambda^\Cat_X$, 
$\rho^\Cat_X$, $\lambda^\Hilb_X$, $\rho^\Hilb_X$ 
the unit isomorphisms of $\Cat$ and, respectively, $\Hilb$. 
Moreover, there are the coherence isomorphisms 
$f_{X,Y}: \For(X) \ot_\Hilb \For(Y) \to \For(X \ot_\Cat Y)$ 
and $f: \one_\Hilb \to \For(\one_\Cat)$ 
of the strong monoidal functor $\For$. 
The first picture in \eqref{eq:hat-A-is-an-algebra} 
translates into 
\begin{align}
&\bigl[ \rho^\Hilb_\L \circ (\id \ot \braL) \circ f_{L,L}^{-1} 
	\circ \For(F) \circ f_{L,L} \circ (\id_\L \ot \ketL) \circ (\rho^\Hilb_\L)^{-1} \\
&\qquad \circ \rho^\Hilb_\L \circ (\id \ot \braL) \circ f_{L,L}^{-1} 
	\circ \For(F) \circ f_{L,L} \circ (\id_\L \ot \ketL) \circ (\rho^\Hilb_\L)^{-1} \bigr] 
	\nonumber\\
&= \bigl[ \rho^\Hilb_\L \circ (\id \ot \braL) \circ f_{L,L}^{-1} 
	\circ \For(F) \circ f_{L,L} \circ (\id_\L \ot \Kp(\L)) \nonumber\\
&\qquad \circ f_{L,L}^{-1} 
	\circ \For(F) \circ f_{L,L} \circ (\id_\L \ot \ketL) \circ (\rho^\Hilb_\L)^{-1} \bigr] \ ,
	\nonumber
\end{align}
where we put 
$|\L\rangle := \B(\C,\L)$ and $\langle\L| := \B(\L,\C)$. 
Using \eqref{eq:semi-regularity-for-braiding}, which we 
write as 
$\K(\L) = \bigl[ \rho^\Hilb_\L \circ (\id_\L \ot \braL) 
	 \circ f_{L,L}^{-1} \circ \For(c_{L,L}) \circ f_{L,L} 
	 \circ (\id_\L \ot \ketL) \circ (\rho^\Hilb_\L)^{-1} \bigr]$, 
we arrive at 
\begin{align}
&\bigl[ \rho^\Hilb_\L \circ (\id \ot \braL) \circ f_{L,L}^{-1} 
	\circ \For(F) \circ f_{L,L} \\
&\qquad \circ (\id_\L \ot \rho^\Hilb_\L) \circ (\id_\L \ot (\id_\L \ot \braL)) \nonumber\\
&\qquad\qquad \circ (\id_\L \ot f_{L,L}^{-1}) \circ (\id_\L \ot \For(c_{L,L}))
	\circ (\id_\L \ot f_{L,L}) \nonumber\\
&\qquad\qquad \circ (\id_\L \ot (\id_\L \ot \ketL)) 
	\circ (\id_\L \ot (\rho^\Hilb_\L)^{-1}) \nonumber\\
&\qquad \circ f_{L,L}^{-1} 
	\circ \For(F) \circ f_{L,L} \circ (\id_\L \ot \ketL) \circ (\rho^\Hilb_\L)^{-1} \bigr] \ ,
	\nonumber
\end{align}
represented by the second picture in \eqref{eq:hat-A-is-an-algebra}. 
To prepare for the application of the Pentagon equation 
\eqref{eq:br-pentagon-1}, we \enquote{pull} several 
maps into the argument of the fibre functor $\For$. 
First, we use the coherence property 
$\alpha^\Hilb_{\L,\L,\one} \circ (\id_\L \ot \rho^\Hilb_\L) 
	= \rho^\Hilb_{\L \ot \L}$ 
to obtain 
\begin{align}
&\bigl[ \rho^\Hilb_\L \circ (\id_\L \ot \rho^\Hilb_\L) 
	\circ (\id \ot (\braL \ot \braL)) \circ \alpha^\Hilb_{\L,\L,\L} 
		\\
&\qquad \circ (f_{L,L}^{-1} \ot \id_\L) 
	\circ (\For(F) \ot \id_\L) \circ (f_{L,L} \ot \id_\L)
	\circ (\alpha^\Hilb_{\L,\L,\L})^{-1} \nonumber\\
&\qquad \circ (\id_\L \ot f_{L,L}^{-1}) \circ (\id_\L \ot \For(c_{L,L}))
	\circ (\id_\L \ot f_{L,L}) \circ \alpha^\Hilb_{\L,\L,\L} \nonumber\\
&\qquad \circ (f_{L,L}^{-1} \ot \id_\L) 
	\circ (\For(F) \ot \id_\L) \circ (f_{L,L} \ot \id_\L)
	\circ (\alpha^\Hilb_{\L,\L,\L})^{-1} \nonumber\\
&\qquad \circ (\id_\L \ot (\ketL \ot \ketL)) \circ (\id_\L \ot (\rho^\Hilb_\L)^{-1}) 
	\circ (\rho^\Hilb_\L)^{-1} \bigr] \ .
	\nonumber
\end{align}
Next, we employ the coherence property 
$f_{L,L\ot L} \circ (\id_\L \ot f_{L,L}) \circ \alpha^\Hilb_{\L,\L,\L}
	= \For(\alpha^\Cat_{L,L,L}) \circ f_{L \ot L,L} \circ (f_{L,L} \ot \id_\L)$, 
and naturality of $f_{X,Y}$, to get 
\begingroup
\allowdisplaybreaks
\begin{align}
&\bigl[ \rho^\Hilb_\L \circ (\id_\L \ot \rho^\Hilb_\L) 
	\circ (\id \ot (\braL \ot \braL)) 
	\circ (\id_\L \ot f_{L,L}^{-1}) \circ f_{L,L \ot L}^{-1} \\
&\quad \circ \For(\alpha^\Cat_{L,L,L} \circ (F \ot \id_L) \circ (\alpha^\Cat_{L,L,L})^{-1}
	\circ (\id_L \ot c_{L,L}) ) \nonumber\\
&\quad \circ \For(\alpha^\Cat_{L,L,L} \circ (F \ot \id_L) \circ (\alpha^\Cat_{L,L,L})^{-1}) 
	\nonumber\\
&\quad \circ f_{L,L \ot L} \circ (\id_\L \ot f_{L,L}) 
	\circ (\id_\L \ot (\ketL \ot \ketL)) \circ (\id_\L \ot (\rho^\Hilb_\L)^{-1}) 
	\circ (\rho^\Hilb_\L)^{-1} \bigr] \ . 
	\nonumber
\end{align}
\endgroup
Now we can apply the Pentagon equation in $\Cat$ 
to arrive at the desired result.

\medskip

\emph{Part 2:} 
The proof is the same as in, e.g.\@ \cite{Baaj-Skandalis:Unitaires}*{Prop.\,1.4}:
Let $\xi,\eta\in\L\setminus\{0\}$. 
Since $F^*(\xi \ot \eta) \neq 0$, there exist $\alpha,\beta\in\L$ 
such that $\skapro{\xi\ot\eta, F(\alpha\ot\beta)} \neq 0$. 
Hence, $(\id \ot \omega_{\eta,\beta})(F)\alpha$ 
is not orthogonal to $\xi$, where $\omega_{\eta,\beta} \in \B(\L)_*$ 
denotes the functional $\omega_{\eta,\beta}(x) = \skapro{\eta, x\beta}$.
\end{proof}

If $F$ is a multiplicative unitary in $\Cat$, then 
$\hat F := c_{L,L}^{-1} F^* c_{L,L}$ is a 
multiplicative unitary in $\Cat^\mathrm{rev}$. 
The previous proposition shows that also $\hatA(\hat F)$ 
is a non-degenerate algebra. 

\begin{df}
Let $\Cat$ be a unitarily braided concrete monoidal $W^*$-category, 
and let $L \in \Cat$ be an object. 
A multiplicative unitary $F \in \U_{\Cat}(L \ot L)$ 
is called \emph{semi-regular}, if the vector space 
\begin{align}
C(F) &:= [(\id \ot \omega)(c_{L,L}^{-1} \circ F)\,|\, \omega \in \B(\L)_*] \ 
\end{align}
contains the space $\Kp(\L)$. 
It is called \emph{regular}, if $C(F)$ 
is equal to $\Kp(\L)$. 
\end{df}

\begin{rmk}\phantomsection\label{rmk:regular-MU}
\begin{enumerate}
\item 
Since we have 
\begin{align}
C(\hat F) &= [(\id \ot \omega)(c_{L,L} \hat F)\,|\, \omega \in \B(\L)_*]
	= [(\id \ot \omega)(F^* c_{L,L})\,|\, \omega \in \B(\L)_*]
	\nonumber\\
&= [(\id \ot \omega)(c_{L,L}^{-1} F)\,|\, \omega \in \B(\L)_*]^*
	= C(F)^* \ ,
\end{align}
a multiplicative unitary $F$ is (semi-)regular, if and only 
if $\hat F$ is (semi-)regular. 
\item 
The condition \eqref{eqn:weaker-goodness-condition} 
for $F$ to be good is equivalent to $C(F) \subseteq \B(\L)$ 
having trivial commutant. 
In particular, semi-regular multiplicative unitaries 
are good. 
Hence, Yetter-Drinfeld categories $\FYDF$ of semi-regular 
multiplicative unitaries $F$ are braided.
\end{enumerate}
\end{rmk}

\begin{lem}
Let $\Cat$ be a semi-regularly braided concrete monoidal $W^*$-category, 
let $L \in \Cat$ be an object and let 
$F \in \U_{\Cat}(L \ot L)$ be a multiplicative unitary. 
Then $C(F) \subseteq \B(\L)$ is a subalgebra. 
If, in addition, $F$ is semi-regular, then $C(F)$ is a $C^*$-algebra.
\end{lem}
\begin{proof}
That $C(F)$ is a subalgebra is seen as follows:
\begin{align}
&[C(F)C(F)] \nonumber\\
&= 
\left[ 
\scalebox{\scalefactor}{
\begin{tikzpicture}[baseline=-\the\dimexpr\fontdimen22\textfont2\relax+1.7cm, line width=.1em]
\draw (0,0) -- (0,.5);
\draw (0,1.5) -- (0,2.25);
\draw (0,3.25) -- (0,3.5);
\draw (1,.25) -- (1,.5);
\draw (1,2) -- (1,2.25);
\draw (0,.5) -- (1,.5);
\draw (0,2.25) -- (1,2.25);
\braid[border height=.0cm,height=1cm,width=1cm] at (0,1.5) s_1;
\braid[border height=.0cm,height=1cm,width=1cm] at (0,3.25) s_1;
\bra{1}{3.25};
\ket{1}{2};
\bra{1}{1.5};
\ket{1}{.25};
\node at (.2,.15) {$L$};
\node at (.2,3.35) {$L$};
\node at (-.2,.5) {$F$};
\node at (-.2,2.25) {$F$};
\end{tikzpicture}} 
\right]
\stackrel{\eqref{eq:semi-regularity-for-braiding}}\subseteq 
\left[ 
\scalebox{\scalefactor}{
\begin{tikzpicture}[baseline=-\the\dimexpr\fontdimen22\textfont2\relax+1.7cm, line width=.1em]
\draw (0,0) -- (0,.5);
\draw (0,1.5) -- (0,2.25);
\draw (0,3.25) -- (0,3.5);
\draw (1,.25) -- (1,.5);
\draw (1,2) -- (1,2.25);
\draw (0,.5) -- (1,.5);
\draw (0,2.25) -- (1,2.25);
\braid[border height=.0cm,height=1cm,width=1cm] at (0,1.5) s_1;
\braid[border height=.0cm,height=1cm,width=1cm] at (0,3.25) s_1;
\braid[border height=.15cm,height=.5cm,width=.5cm] at (1,2.15) s_1;
\bra{1}{3.25};
\bra{1.5}{2.15};
\ket{1.5}{1.35};
\ket{1}{.25};
\end{tikzpicture}} 
\right]
=
\left[ 
\scalebox{\scalefactor}{
\begin{tikzpicture}[baseline=-\the\dimexpr\fontdimen22\textfont2\relax+1.7cm, line width=.1em]
\draw (0,0) -- (0,1.5);
\draw (0,3) -- (0,3.5);
\draw (1,.25) -- (1,1.5);
\draw (2,.5) -- (2,1.5);
\draw (0,.75) -- (1,.75);
\draw (1,1.25) -- (2,1.25);
\braid[border height=.0cm,height=.5cm,width=1cm] at (0,3) s_1 s_2 s_1;
\bra{1}{3.1};
\bra{2}{2.9};
\ket{2}{.5};
\ket{1}{.25};
\end{tikzpicture}} 
\right]
\stackrel{\eqref{eq:br-pentagon-1}}=
\left[ 
\scalebox{\scalefactor}{
\begin{tikzpicture}[baseline=-\the\dimexpr\fontdimen22\textfont2\relax+1.7cm, line width=.1em]
\draw (0,0) -- (0,1.5);
\draw (0,3) -- (0,3.5);
\draw (1,.25) -- (1,1.5);
\draw (2,.25) -- (2,1.5);
\draw (0,1.35) -- (1,1.35);
\draw[color=white, line width=.7em] (0.5,1) -- (1.5,1);
\draw (0,1) -- (2,1);
\draw (1,.65) -- (2,.65);
\braid[border height=.0cm,height=.5cm,width=1cm] at (0,3) s_1 s_2 s_1;
\bra{1}{3.1};
\bra{2}{2.9};
\ket{2}{.25};
\ket{1}{.25};
\end{tikzpicture}} 
\right]
=
\left[ 
\scalebox{\scalefactor}{
\begin{tikzpicture}[baseline=-\the\dimexpr\fontdimen22\textfont2\relax+1.7cm, line width=.1em]
\draw (0,0) -- (0,2);
\draw (0,3) -- (0,3.5);
\draw (1,.25) -- (1,.75);
\draw (1,1.25) -- (1,2);
\draw (2,.25) -- (2,.75);
\draw (2,1.25) -- (2,2);
\draw (1,2.5) -- (2,2.5);
\draw (0,1.625) -- (1,1.625);
\draw (1,.65) -- (2,.65);
\braid[border height=.0cm,height=.5cm,width=1cm] at (0,3) s_2 s_1;
\braid[border height=.0cm,height=.5cm,width=1cm] at (1,1.25) s_1;
\bra{1}{3.1};
\bra{2}{3.1};
\ket{2}{.25};
\ket{1}{.25};
\end{tikzpicture}} 
\right]
\nonumber\\
&\stackrel{\eqref{eq:remove-unitaries}}= C(F) \ .
\end{align}
If $F$ is semi-regular, then we have 
\begin{align}
&[C(F)C(F)^*] \nonumber\\
&= 
\left[ 
\scalebox{\scalefactor}{
\begin{tikzpicture}[baseline=-\the\dimexpr\fontdimen22\textfont2\relax+1.7cm, line width=.1em]
\draw (0,0) -- (0,.25);
\draw (0,1.4) -- (0,2.25);
\draw (0,3.25) -- (0,3.5);
\draw (1,2) -- (1,2.25);
\draw[dashed] (0,1.25) -- (1,1.25);
\draw (0,2.25) -- (1,2.25);
\braid[border height=.15cm,height=1cm,width=1cm] at (0,1.4) s_1^{-1};
\braid[border height=.0cm,height=1cm,width=1cm] at (0,3.25) s_1;
\bra{1}{3.25};
\ket{1}{2};
\bra{1}{1.5};
\ket{1}{.25};
\node at (.2,.15) {$L$};
\node at (.2,3.35) {$L$};
\node at (-.25,1.25) {$F^*$};
\node at (-.2,2.25) {$F$};
\end{tikzpicture}} 
\right]
\stackrel[\text{reg.}]{\text{semi-}}\subseteq 
\left[ 
\scalebox{\scalefactor}{
\begin{tikzpicture}[baseline=-\the\dimexpr\fontdimen22\textfont2\relax+1.7cm, line width=.1em]
\draw (0,0) -- (0,.25);
\draw (0,1.4) -- (0,2.25);
\draw (0,3.25) -- (0,3.5);
\draw (1,1.25) -- (1,1.5);
\draw (2,1.25) -- (2,1.5);
\draw[dashed] (0,1.25) -- (1,1.25);
\draw (0,2.25) -- (1,2.25);
\draw (1,1.5) -- (2,1.5);
\braid[border height=.0cm,height=1cm,width=1cm] at (0,3.25) s_1;
\braid[border height=.0cm,height=.75cm,width=1cm] at (1,2.25) s_1;
\braid[border height=.15cm,height=1cm,width=1cm] at (0,1.4) s_1^{-1};
\bra{1}{3.25};
\bra{2}{2.35};
\ket{2}{1.25};
\ket{1}{.25};
\end{tikzpicture}} 
\right]
= 
\left[ 
\scalebox{\scalefactor}{
\begin{tikzpicture}[baseline=-\the\dimexpr\fontdimen22\textfont2\relax+1.7cm, line width=.1em]
\draw (0,0) -- (0,.25);
\draw (0,1.4) -- (0,2.25);
\draw (0,3.25) -- (0,3.5);
\draw (1,1.25) -- (1,2.25);
\draw (2,1) -- (2,2.25);
\draw[dashed] (0,1.25) -- (1,1.25);
\draw[color=white, line width=.7em] (.5,2) -- (1.5,2);
\draw (0,2) -- (2,2);
\draw (1,1.625) -- (2,1.625);
\braid[border height=.0cm,height=.5cm,width=1cm] at (0,3.25) s_1 s_2;
\braid[border height=.15cm,height=1cm,width=1cm] at (0,1.4) s_1^{-1};
\bra{1}{3.3};
\bra{2}{3.1};
\ket{2}{1};
\ket{1}{.25};
\end{tikzpicture}} 
\right]
\stackrel{\eqref{eq:br-pentagon-1}}= 
\left[ 
\scalebox{\scalefactor}{
\begin{tikzpicture}[baseline=-\the\dimexpr\fontdimen22\textfont2\relax+1.7cm, line width=.1em]
\draw (0,0) -- (0,.25);
\draw (0,1.4) -- (0,2.25);
\draw (0,3.25) -- (0,3.5);
\draw (1,1.25) -- (1,2.25);
\draw (2,1) -- (2,2.25);
\draw[dashed] (0,1.9) -- (1,1.9);
\draw (1,1.5) -- (2,1.5);
\braid[border height=.0cm,height=.5cm,width=1cm] at (0,3.25) s_1 s_2;
\braid[border height=.15cm,height=1cm,width=1cm] at (0,1.4) s_1^{-1};
\bra{1}{3.3};
\bra{2}{3.1};
\ket{2}{1};
\ket{1}{.25};
\end{tikzpicture}} 
\right]
= 
\left[ 
\scalebox{\scalefactor}{
\begin{tikzpicture}[baseline=-\the\dimexpr\fontdimen22\textfont2\relax+1.4cm, line width=.1em]
\draw (0,-.25) -- (0,1);
\draw (0,2) -- (0,3);
\draw (1,.75) -- (1,1);
\draw (1,2.5) -- (1,2.75);
\draw (1.5,.75) -- (1.5,2);
\draw (1.5,2.5) -- (1.5,2.75);
\draw[dashed] (1,2.5) -- (1.5,2.5);
\draw (0,1) -- (1,1);
\braid[border height=.0cm,height=.5cm,width=.5cm] at (1,2.5) s_1^{-1};
\braid[border height=.0cm,height=1cm,width=1cm] at (0,2) s_1;
\braid[border height=.25cm,height=.5cm,width=.5cm] at (1,1) s_1;
\bra{1}{2.75};
\bra{1.5}{2.75};
\ket{1}{0};
\ket{1.5}{0};
\end{tikzpicture}} 
\right]
\nonumber\\
&\stackrel{\eqref{eq:remove-unitaries}}= C(F) 
\end{align}
and 
\begin{align}
&[C(F)^*C(F)] \nonumber\\
&\quad= 
\left[ 
\scalebox{\scalefactor}{
\begin{tikzpicture}[baseline=-\the\dimexpr\fontdimen22\textfont2\relax+1.7cm, line width=.1em]
\draw (0,0) -- (0,.5);
\draw (0,1.5) -- (0,2);
\draw (0,3) -- (0,3.5);
\draw (1,.25) -- (1,.5);
\draw (1,3) -- (1,3.25);
\draw (0,.5) -- (1,.5);
\draw[dashed] (0,3) -- (1,3);
\braid[border height=.0cm,height=1cm,width=1cm] at (0,3) s_1^{-1};
\braid[border height=.0cm,height=1cm,width=1cm] at (0,1.5) s_1;
\bra{1}{3.25};
\ket{1}{2};
\bra{1}{1.5};
\ket{1}{.25};
\node at (.2,.15) {$L$};
\node at (-.2,.5) {$F$};
\node at (-.3,3) {$F^*$};
\end{tikzpicture}} 
\right]
\stackrel[\text{reg.}]{\text{semi-}}\subseteq 
\left[ 
\scalebox{\scalefactor}{
\begin{tikzpicture}[baseline=-\the\dimexpr\fontdimen22\textfont2\relax+1.7cm, line width=.1em]
\draw (0,0) -- (0,.5);
\draw (0,1.5) -- (0,2);
\draw (0,3) -- (0,3.5);
\draw (1,.25) -- (1,.5);
\draw (1,3) -- (1,3.25);
\draw (1.75,1.25) -- (1.75,1.5);
\draw (0,.5) -- (1,.5);
\draw (1,1.5) -- (1.75,1.5);
\draw[dashed] (0,3) -- (1,3);
\braid[border height=.0cm,height=1cm,width=1cm] at (0,3) s_1^{-1};
\braid[border height=.0cm,height=.5cm,width=.75cm] at (1,2) s_1;
\braid[border height=.0cm,height=1cm,width=1cm] at (0,1.5) s_1;
\bra{1}{3.25};
\bra{1.75}{2.1};
\ket{1.75}{1.25};
\ket{1}{.25};
\end{tikzpicture}} 
\right]
= 
\left[ 
\scalebox{\scalefactor}{
\begin{tikzpicture}[baseline=-\the\dimexpr\fontdimen22\textfont2\relax+1.7cm, line width=.1em]
\draw (0,-.25) -- (0,2);
\draw (0,3) -- (0,3.5);
\draw (1,1) -- (1,2);
\draw (2,1) -- (2,2);
\draw[color=white, line width=.7em] (.5,1.1) -- (1.5,1.1);
\draw (0,1.1) -- (2,1.1);
\draw (0,1.5) -- (1,1.5);
\draw[dashed] (1,1.9) -- (2,1.9);
\braid[border height=.15cm,height=.5cm,width=1cm] at (0,3.15) s_2 s_1;
\braid[border height=.0cm,height=1cm,width=1cm] at (1,1) s_1^{-1};
\bra{1}{3.25};
\bra{2}{3.25};
\ket{1}{0};
\ket{2}{0};
\end{tikzpicture}} 
\right]
\stackrel{\eqref{eq:br-pentagon-1}}= 
\left[ 
\scalebox{\scalefactor}{
\begin{tikzpicture}[baseline=-\the\dimexpr\fontdimen22\textfont2\relax+1.7cm, line width=.1em]
\draw (0,-.25) -- (0,2);
\draw (0,3) -- (0,3.5);
\draw (1,1) -- (1,2);
\draw (2,1) -- (2,2);
\draw (0,1.5) -- (1,1.5);
\draw[dashed] (1,1) -- (2,1);
\braid[border height=.15cm,height=.5cm,width=1cm] at (0,3.15) s_2 s_1;
\braid[border height=.0cm,height=1cm,width=1cm] at (1,1) s_1^{-1};
\bra{1}{3.25};
\bra{2}{3.25};
\ket{1}{0};
\ket{2}{0};
\end{tikzpicture}} 
\right]
\stackrel{\eqref{eq:remove-unitaries}}= C(F) \ .
\end{align}
Now, if a norm-closed subalgebra $C$ of $\B(\L)$ satisfies 
$C^*C \subseteq C$ and $CC^* \subseteq C$, then it is 
already a $C^*$-algebra, cf.\@ \cite{Timmermann:An_inv_to_QG_and_duality}*{Lem.\,7.3.7}. 
Namely, for every $a \in C$ there 
exists a sequence $(p_n)_{n \in \N}$ of polynomials such that 
$a = \lim_n a p_n(a^*a)$. Then 
$a^* = \lim_n p_n(a^*a) a^* \in C^*CC^* \subseteq C$. 
\end{proof}

\begin{prop}\label{prop:hat-AF-is-a-C*-alg}
Let $\Cat$ be a semi-regularly braided 
concrete monoidal $W^*$-category, 
and let $L \in \Cat$ be an object. 
If $F \in \U_{\Cat}(L \ot L)$ is a 
multiplicative unitary and $C(F) = C(F)^*$ 
\textnormal{(}in particular, if $F$ is semi-regular\textnormal{)}, 
then $\hatA(F)$ is a $C^*$-algebra. 
\end{prop}
\begin{proof}
We have: 
\begin{align}
&[(F \ot \id_\L)(\id_\L \ot C(F))(F^* \ot \id_\L)] 
\nonumber\\
&\hspace{7em}= \left[ 
\scalebox{\scalefactor}{
\begin{tikzpicture}[baseline=-\the\dimexpr\fontdimen22\textfont2\relax, 
	line width=.1em, x=1cm, y=1cm]
\draw (0,-1.4) -- (0,1.4);
\braid[border height=0cm,height=1cm,width=1cm] at (1,.7) s_1;
\draw (1,.7) -- (1,1);
\draw (1,-1) -- (1,-.2);
\draw (2,-.7) -- (2,-.2);
\draw (1,-.3) -- (2,-.3);		
\draw (0,.7) -- (1,.7);		
\draw[dashed] (0,-.7) -- (1,-.7);		
\bra{1}{1};
\ket{1}{-1};
\bra{2}{.7};
\ket{2}{-.7};
\end{tikzpicture}} 
\right]
\stackrel{\eqref{eq:br-pentagon-1}}= 
\left[ 
\scalebox{\scalefactor}{
\begin{tikzpicture}[baseline=-\the\dimexpr\fontdimen22\textfont2\relax, line width=.1em]
\draw (0,-1.4) -- (0,1.4);
\braid[border height=0cm,height=1cm,width=1cm] at (1,1) s_1;
\draw (1,-1) -- (1,0);
\draw (2,-1) -- (2,0);
\draw (1,-.5) -- (2,-.5);		
\draw[dashed] (0,0) -- (1,0);		
\bra{1}{1};
\ket{1}{-1};
\bra{2}{1};
\ket{2}{-1};
\end{tikzpicture}} 
\right]
\stackrel{\eqref{eq:remove-unitaries}}= \hatA(F)^* \ , 
\end{align}
where the left-hand side is an obviously self-adjoint set. 
Hence, $\hatA(F)^* = \hatA(F) \subseteq \B(\L)$ is a self-adjoint 
norm-closed subset. As we know already that $\hatA(F)$ is an algebra, 
it follows that $\hatA(F)$ is a $C^*$-algebra. 
\end{proof}

\subsection{The monoidal category of $C^*$-algebras $\hatA(\Cat)$}
\label{ssec:category-hat-A(C)}

So far, we have seen that slices over the right leg of a 
semi-regular multiplicative unitary form a $C^*$-algebra 
$\hatA(F)$. 
In the next section we will show that 
$\hatA(F)$ comes equipped with a 
braided bialgebra structure. 
Here, we prepare for this and 
introduce a suitable monoidal category 
of $C^*$-algebras 
(see Definition \ref{df:monoidal-cat-of-C*-algebras}), 
which contains $\hatA(F)$ as an object.
We insist on regularity from now on, instead of just 
semi-regularity. 

\medskip

Let $\Cat$ be a regularly braided concrete monoidal $W^*$-category. 
Given objects $H,K \in \Cat$ and a morphism 
$X \in \End_\Cat(H \ot K)$, we denote by 
$\hatA(X)$ the closed subspace 
$[(\id \ot \omega)(X)\,|\, \omega \in \B(\K)_*]$ 
of $\B(\H)$. 
By $\hatA(\Cat)$ we denote the category of pairs 
$(H,A)$, where $H \in \Cat$ is an object and $A \subseteq \B(\H)$ 
is a $C^*$-subalgebra of the form $A = \hatA(X)$ 
for some morphism $X$. 
As morphisms $(H_1,A_1) \to (H_2,A_2)$ 
in $\hatA(\Cat)$ we take all 
non-degenerate $^*$-homomorphisms $f:\, A_1 \to \Malg(A_2)$ 
which are of the form 
\begin{align}\label{eq:morphisms-in-hatA(C)}
A_1 \ni a \mapsto V(\id_K \ot a)V^* \ ,
\end{align}
for some $K \in \Cat$ and some isometry 
$V \in \Hom_\Cat(K \ot H_1, H_2)$. 
By regularity of the braiding on $\Cat$, for each 
$H \in \Cat$ we have that $(H,\Kp(\H))$ 
is an object of $\hatA(\Cat)$: just choose $X = c_{H,H}$. 
If $F \in \U_\Cat(L \ot L)$ is regular, then 
$(L,\hatA(F))$ is an object of $\hatA(\Cat)$, 
as desired.

Before defining a tensor product on $\hatA(\Cat)$, 
we need the following lemma.

\begin{lem}
\sloppy
Let $\Cat$ be a regularly braided concrete monoidal $W^*$-category 
and let $(H_1,A_1),(H_2,A_2) \in \hatA(\Cat)$ be objects. 
Then 
\begin{align}\label{eq:hbt-on-algebras}
A_1 \hbt A_2 := \bigl[c_{H_2,H_1}(\id_{\H_2} \ot A_1)c_{H_2,H_1}^{-1} 
	\cdot (\id_{\H_1} \ot A_2)\bigr] \subseteq \B(\H_1 \ot \H_2)
\end{align}
is a crossed product of $A_1$ and $A_2$ 
with injections 
$\alpha: A_1 \ni a \mapsto c_{H_2,H_1}(\id_{\H_2} \ot a)c_{H_2,H_1}^{-1}$ and 
$\beta: A_2 \ni a \mapsto \id_{\H_1} \ot a$. 
Moreover, $(H_1 \ot H_2, A_1 \hbt A_2)$ is an object 
of $\hatA(\Cat)$ and $\alpha,\beta$ are morphisms 
in $\hatA(\Cat)$.
\end{lem}

\fussy
In the notation \enquote{$A_1 \hbt A_2$} we 
deliberately suppressed the dependence on 
$H_1,H_2$.

\begin{proof}
We first have to show that $A_1 \hbt A_2$ is 
a $C^*$-algebra. 
Since $(H_1,A_1)$ and $(H_2,A_2)$ are objects 
of $\hatA(\Cat)$, we find objects $K_1,K_2 \in \Cat$ 
and morphisms $X_i \in \End_\Cat(H_i \ot K_i)$, $i=1,2$, 
such that $A_1 = \hatA(X_1)$ and $A_2 = \hatA(X_2)$. 
Then 
\begin{align}\label{eq:lem-on-braided-cross-product--1}
&\bigl[\alpha(A_1) \cdot \beta(A_2)\bigr]
\nonumber\\
&=
\left[ 
\scalebox{\scalefactor}{
\begin{tikzpicture}[baseline=-\the\dimexpr\fontdimen22\textfont2\relax+1.8cm, line width=.1em]
\draw (0,.25) -- (0,1.5);
\draw (1,.25) -- (1,1.5);
\draw (2,.5) -- (2,1.5);
\draw (2,2) -- (2,3);
\draw (1,1) -- (2,1);
\draw (1,2.5) -- (2,2.5);
\braid[border height=.0cm,height=1cm,width=1cm] at (0,3.5) s_1^{-1} s_1;
\bra{2}{3};
\ket{2}{2};
\bra{2}{1.5};
\ket{2}{.5};
\node at (1.5,1.2) {$X_2$};
\node at (1.5,2.7) {$X_1$};
\node[anchor=west,inner sep=.15em] at (0,.4) {$H_1$};
\node[anchor=west,inner sep=.15em] at (1,.4) {$H_2$};
\end{tikzpicture}} 
\right]
= 
\left[ 
\scalebox{\scalefactor}{
\begin{tikzpicture}[baseline=-\the\dimexpr\fontdimen22\textfont2\relax+1.8cm, line width=.1em]
\draw (0,.25) -- (0,1.5);
\draw (1,.25) -- (1,1.5);
\draw (2,.5) -- (2,1.5);
\draw (2,2) -- (2,3);
\draw (1,1) -- (2,1);
\draw (1,2.5) -- (2,2.5);
\braid[border height=.0cm,height=1cm,width=1cm] at (0,3.5) s_1^{-1} s_1;
\braid[border height=.15cm,height=.5cm,width=.5cm] at (2,2.15) s_1^{-1};
\bra{2}{3};
\bra{2.5}{2.15};
\ket{2.5}{1.35};
\ket{2}{.5};
\node at (1.5,1.2) {$X_2$};
\node at (1.5,2.7) {$X_1$};
\end{tikzpicture}} 
\right]
\stackrel{\eqref{eq:remove-unitaries}}= 
\left[ 
\scalebox{\scalefactor}{
\begin{tikzpicture}[baseline=-\the\dimexpr\fontdimen22\textfont2\relax+2cm, line width=.1em]
\draw (0,.5) -- (0,1.5);
\draw (1,.5) -- (1,1.5);
\draw (1,1.5) -- (2,1.5);
\draw (1,2.5) -- (2,2.5);
\braid[border height=.0cm,height=1cm,width=1cm] at (0,3.5) s_1^{-1} s_1;
\braid[border height=.0cm,height=1cm,width=.5cm] at (2,2.5) s_1^{-1};
\braid[border height=.15cm,height=.5cm,width=.5cm] at (2,3.15) s_1;
\braid[border height=.15cm,height=.5cm,width=.5cm] at (2,1.65) s_1;
\bra{2}{3.15};
\bra{2.5}{3.15};
\ket{2.5}{.85};
\ket{2}{.85};
\node at (1.5,1.7) {$X_2$};
\node at (1.5,2.7) {$X_1$};
\end{tikzpicture}} 
\right]
\stackrel{\eqref{eq:remove-unitaries}}= 
\left[ 
\scalebox{\scalefactor}{
\begin{tikzpicture}[baseline=-\the\dimexpr\fontdimen22\textfont2\relax-1.9cm, line width=.1em, y=-1cm]
\draw (0,.25) -- (0,1.5);
\draw (1,.25) -- (1,1.5);
\draw (2,.5) -- (2,1.5);
\draw (2,2) -- (2,3);
\draw (1,1) -- (2,1);
\draw (1,2.5) -- (2,2.5);
\braid[border height=.0cm,height=1cm,width=1cm] at (0,1.5) s_1^{-1} s_1;
\braid[border height=.15cm,height=.5cm,width=.5cm] at (2,1.35) s_1^{-1};
\ket{2}{3};
\ket{2.5}{2.15};
\bra{2.5}{1.35};
\bra{2}{.5};
\node at (1.5,.8) {$X_2$};
\node at (1.5,2.3) {$X_1$};
\end{tikzpicture}} 
\right]
\\
&=\bigl[\beta(A_2) \cdot \alpha(A_1) \bigr] \ 
\nonumber
\end{align}
shows first, that $A_1 \hbt A_2 \subseteq \B(\H_1 \ot \H_2)$ 
is a subalgebra. 
\sloppy
Since both $\alpha(A_1)$ 
and $\beta(A_2)$ are $^*$-invariant subsets of $\B(\H_1 \ot \H_2)$, 
it also shows that $A_1 \hbt A_2$ is $^*$-invariant, 
hence, a $C^*$-subalgebra. 
Thirdly, we see from 
\eqref{eq:lem-on-braided-cross-product--1} 
that $A_1 \hbt A_2 = \hatA(X)$ for 
$X=(c_{H_2,H_1} \ot c_{K_2,K_1}^{-1})\circ(\id_{H_2} \ot X_1 \ot \id_{K_2})
	\circ(c_{H_2,H_1}^{-1} \ot c_{K_2,K_1})\circ(\id_{H_1} \ot X_2 \ot \id_{K_1}))$, 
hence, $(H_1 \ot H_2, A_1 \hbt A_2)$ is an object 
of $\hatA(\Cat)$. 
Finally, it is easily checked -- again using 
\eqref{eq:lem-on-braided-cross-product--1} -- 
that $\alpha$ and $\beta$ are non-degenerate 
$^*$-homomorphisms $A_i \to \Malg(A_1 \hbt A_2)$, 
$i = 1,2$, 
and they are clearly of the form 
\eqref{eq:morphisms-in-hatA(C)}.
\end{proof}

\fussy
We can now define a tensor product $\hbt$ on $\hatA(\Cat)$ 
for objects $(H_1,A_1),(H_2,A_2) \in \hatA(\Cat)$ by 
$(H_1,A_1) \hbt (H_2,A_2) := (H_1 \ot H_2, A_1 \hbt A_2)$. 
The definition of $\hbt$ for morphisms is done in the following 
lemma.

\begin{lem}
Let $\Cat$ be a regularly braided concrete monoidal $W^*$-category 
and let $f \in \Hom_{\hatA(\Cat)}((H_1,A_1), (H_1',A_1'))$ and 
$g \in \Hom_{\hatA(\Cat)}((H_2,A_2), (H_2',A_2'))$ 
be morphisms. 
Then 
\begin{align}\label{eq:hbt-on-morphisms}
(f \hbt g)(\alpha(a_1)\cdot \beta(a_2)) 
	:= \alpha(f(a_1)) \cdot \beta(g(a_2)) \ ,
	\qquad a_1 \in A_1,\ a_2 \in A_2 \ ,
\end{align}
defines a morphism 
$f \hbt g \in \Hom_{\hatA(\Cat)}((H_1,A_1) \hbt (H_2,A_2), 
	(H_1',A_1') \hbt (H_2',A_2'))$. 
Moreover, for composable morphisms 
$f,f'$ and $g, g'$ we have 
$(f' \hbt g') \circ (f \hbt g) = (f' \circ f) \hbt (g' \circ g)$. 
\end{lem}
\begin{proof}
Let $V_f \in \Hom_\Cat(K_f \ot H_1, H_1')$ 
and $V_g \in \Hom_\Cat(K_g \ot H_2, H_2')$ 
be isometries such that 
$f(a) = V_f(\id_{\K_f} \ot a)V_f^*$ and 
$g(b) = V_g(\id_{\K_g} \ot b)V_g^*$ for 
$a \in A_1$ and $b \in A_2$. 
Then 
\begin{align}
\alpha(f(a)) \cdot \beta(g(b)) 
= 
\scalebox{\scalefactor}{
\begin{tikzpicture}[baseline=-\the\dimexpr\fontdimen22\textfont2\relax+2.875cm, line width=.1em]
\draw (0,0) -- (0,2.25);
\draw (0,3) -- (0,5);
\draw (1,0) -- (1,.5);
\draw (.7,.5) -- (.7,2);
\draw (1.3,.5) -- (1.3,2);
\draw (.7,3.25) -- (.7,4.75);
\draw (1.3,3.25) -- (1.3,4.75);
\braid[border height=.0cm,height=.75cm,width=1cm] at (0,5.75) s_1^{-1};
\braid[border height=.0cm,height=.75cm,width=1cm] at (0,3) s_1;
\node[draw,fill=white,inner sep=.15em,minimum width=.8cm] at (1,4.75) {$V_f$};
\node[draw,circle,fill=white,inner sep=.15em] at (1.3,4) {$a$};
\node[draw,fill=white,inner sep=.15em,minimum width=.8cm] at (1,3.25) {$V_f^*$};
\node[draw,fill=white,inner sep=.15em,minimum width=.8cm] at (1,2) {$V_g$};
\node[draw,circle,fill=white,inner sep=.15em] at (1.3,1.25) {$b$};
\node[draw,fill=white,inner sep=.15em,minimum width=.8cm] at (1,.5) {$V_g^*$};
\node[anchor=north] at (0,0) {$H_1'$};
\node[anchor=north] at (1,0) {$H_2'$};
\end{tikzpicture}} 
=
\scalebox{\scalefactor}{
\begin{tikzpicture}[baseline=-\the\dimexpr\fontdimen22\textfont2\relax+2.125cm, line width=.1em]
\draw (0,0) -- (0,.5);
\draw (-.3,.5) -- (-.3,3.75);
\draw (.25,1.25) -- (.25,3);
\draw (.75,2) -- (.75,2.5);
\draw (0,3.75) -- (0,4.25);
\draw (1,0) -- (1,.5);
\draw (.75,1.25) -- (.75,1.5);
\draw (1.3,.5) -- (1.3,1.5);
\draw (1.3,3) -- (1.3,3.75);
\draw (1,3.75) -- (1,4.25);
\braid[border height=.0cm,height=.5cm,width=.5cm] at (.25,3.5) s_1^{-1};
\braid[border height=.0cm,height=.5cm,width=.55cm] at (.75,3) s_1^{-1};
\braid[border height=.0cm,height=.5cm,width=.55cm] at (.75,2) s_1;
\braid[border height=.0cm,height=.5cm,width=.5cm] at (.25,1.25) s_1;
\node[draw,fill=white,inner sep=.15em,minimum width=.8cm] at (0,3.75) {$V_f$};
\node[draw,fill=white,inner sep=.15em,minimum width=.8cm] at (1,3.75) {$V_g$};
\node[draw,circle,fill=white,inner sep=.15em] at (1.3,2.25) {$a$};
\node[draw,circle,fill=white,inner sep=.15em] at (1.3,1.25) {$b$};
\node[draw,fill=white,inner sep=.15em,minimum width=.8cm] at (0,.5) {$V_f^*$};
\node[draw,fill=white,inner sep=.15em,minimum width=.8cm] at (1,.5) {$V_g^*$};
\node[anchor=north] at (0,0) {$H_1'$};
\node[anchor=north] at (1,0) {$H_2'$};
\end{tikzpicture}} 
\ ,
\qquad a \in A_1,\ b \in A_2 \ .
\end{align}
This shows that \eqref{eq:hbt-on-morphisms} 
is well-defined and extends to a $^*$-homomorphism 
$f \hbt g: A_1 \hbt A_2 \to \B(\H_1' \ot \H_2')$ 
of the form \eqref{eq:morphisms-in-hatA(C)}. 
Moreover, 
\begin{align}
&\bigl[ (f \hbt g)(A_1 \hbt A_2)\cdot (A_1' \hbt A_2') \bigr]
\stackrel{\eqref{eq:hbt-on-algebras}}= 
	\bigl[ (f \hbt g)([\alpha(A_1) \beta(A_2)]) 
		\cdot [\alpha(A_1') \beta(A_2')] \bigr]
	\nonumber\\
&= \bigl[ \alpha(f(A_1)) \beta(g(A_2)) \cdot [\beta(A_2') \alpha(A_1')] \bigr]
= \bigl[ \alpha(f(A_1)) \underbrace{[\beta(g(A_2)) 
		\cdot \beta(A_2')]}_{= \beta(A_2')} \alpha(A_1') \bigr]
	\nonumber\\
&= \bigl[ \underbrace{[\alpha(f(A_1)) \alpha(A_1')]}_{= \alpha(A_1')} \beta(A_2') \bigr]
= A_1' \hbt A_2' 
\end{align}
implies that the image of $f \hbt g$ is in 
$\Malg(A_1' \hbt A_2') \subseteq \B(\H_1' \ot \H_2')$ 
and that $f \hbt g$ is non-degenerate. 
That $\hbt$ is functorial is immediate from its 
definition \eqref{eq:hbt-on-morphisms}. 
\end{proof}

There is another tensor product $\habt$ on $\hatA(\Cat)$, 
which is obtained in a similar way as $\hbt$, but 
using the inverse braiding instead of 
the braiding $c$. 
For example, on objects $(H_1,A_1)$ and $(H_2,A_2)$ of 
$\hatA(\Cat)$ we have 
$A_1 \habt A_2 := \bigl[c_{H_1,H_2}^{-1}(\id_{\H_2} \ot A_1)c_{H_1,H_2} 
	\cdot (\id_{\H_1} \ot A_2)\bigr]$.

\subsection{$\hatA(F)$ as a $C^*$-bialgebra in $\hatA(\Cat)$}
\label{ssec:hat-A(F)-as-C*-bialgebra}

Recall that by $\ot_\mathrm{min}$ we denote the minimal (spatial) 
$C^*$-tensor product. 
We have the following relation between the tensor 
products $\ot_\mathrm{min}, \hbt$ and $\habt$:

\begin{lem}\label{lem:on-braided-tensor-products}
Let $\Cat$ be a regularly braided concrete monoidal $W^*$-category, 
let $(K,A) \in \hatA(\Cat)$ and $H \in \Cat$ be 
an object with underlying Hilbert space $\H$. 
Then we have 
\begin{align}
A \ot_\mathrm{min} \Kp(\H) = A \hbt \Kp(\H) 
	= A \habt \Kp(\H)
\end{align}
as subspaces of $\B(\K \ot \H)$. 
\end{lem}
\begin{proof}
Let $X$ be a morphism in $\Cat$ such that $A = \hatA(X)$. 
We have 
\begin{align}\label{eq:lem-on-braided-tens-prod--pf--1}
&A \ot_\mathrm{min} \Kp(\H) 
\nonumber\\
&= 
\left[ 
\scalebox{\scalefactor}{
\begin{tikzpicture}[baseline=-\the\dimexpr\fontdimen22\textfont2\relax+1.9cm, line width=.1em]
\draw (0,.5) -- (0,3.5);
\draw (1,.5) -- (1,1);
\draw (1,1.5) -- (1,2.5);
\draw (1,3) -- (1,3.5);
\draw (0,2) -- (1,2);
\ket{1}{3};
\bra{1}{2.5};
\ket{1}{1.5};
\bra{1}{1};
\node[anchor=south,inner sep=.2em] at (.5,2) {$X$};
\end{tikzpicture}} 
\right]
= 
\left[ 
\scalebox{\scalefactor}{
\begin{tikzpicture}[baseline=-\the\dimexpr\fontdimen22\textfont2\relax+1.9cm, line width=.1em]
\draw (0,.5) -- (0,3.5);
\draw (1,.5) -- (1,1);
\draw (1,1.5) -- (1,2.5);
\draw (1,3) -- (1,3.5);
\draw (0,2) -- (1,2);
\braid[border height=.15cm,height=.5cm,width=.5cm] at (1,3.15) s_1;
\bra{1.5}{3.15};
\ket{1.5}{2.35};
\ket{1}{1.5};
\bra{1}{1};
\node[anchor=south,inner sep=.2em] at (.5,2) {$X$};
\end{tikzpicture}} 
\right]
\stackrel{\eqref{eq:remove-unitaries}}= 
\left[ 
\scalebox{\scalefactor}{
\begin{tikzpicture}[baseline=-\the\dimexpr\fontdimen22\textfont2\relax+1.9cm, line width=.1em]
\draw (0,.5) -- (0,3.5);
\draw (1,.5) -- (1,1);
\draw (1,3) -- (1,3.5);
\draw (0,2.25) -- (1,2.25);
\braid[border height=.15cm,height=.75cm,width=.5cm] at (1,3.15) s_1 s_1^{-1};
\bra{1.5}{3.15};
\ket{1.5}{1.5};
\ket{1}{1.5};
\bra{1}{1};
\node[anchor=south,inner sep=.2em] at (.5,2.25) {$X$};
\end{tikzpicture}} 
\right]
= 
\left[ 
\scalebox{\scalefactor}{
\begin{tikzpicture}[baseline=-\the\dimexpr\fontdimen22\textfont2\relax+1.9cm, line width=.1em]
\draw (0,.5) -- (0,1.5);
\draw (1,.5) -- (1,1);
\draw (2,2) -- (2,3);
\draw (1,2.5) -- (2,2.5);
\braid[border height=.0cm,height=1cm,width=1cm] at (0,3.5) s_1^{-1} s_1;
\bra{2}{3};
\ket{2}{2};
\ket{1}{1.5};
\bra{1}{1};
\node[anchor=south,inner sep=.2em] at (1.5,2.5) {$X$};
\end{tikzpicture}} 
\right]
= 
A \hbt \Kp(\H) 
\end{align}
as vector spaces. 
The statement with $\hbt$ replaced by 
$\habt$ follows by replacing $\Cat$ by 
$\Cat^\mathrm{rev}$. 
\end{proof}

We have the following analogue of 
\cite{Baaj-Skandalis:Unitaires}*{Prop.\,3.6}:

\begin{thm}\label{thm:F-as-multiplier}
Let $\Cat$ be a regularly braided concrete monoidal $W^*$-category, 
let $L \in \Cat$ be an object with underlying 
Hilbert space $\L$, and let $F \in \U_{\Cat}(L \ot L)$ 
be a regular multiplicative unitary. 
Then we have 
\begin{enumerate}
\item $F \in \Malg(\hatA(F) \hbt \hatA(\hat F))$ 
\item $[(\hatA(F) \hbt \id_\L)F(\id_\L \hbt \hatA(\hat F))] = \hatA(F) \hbt \hatA(\hat F)$. 
\end{enumerate}
\end{thm}
\begin{proof}
\emph{Part 1:} 
\sloppy
We divide the proof into three steps, first showing that 
$F \in \Malg(\hatA(F) \hbt \Kp(\L))$, then 
$F \in \Malg(\Kp(\L) \hbt \hatA(\hat F))$, before proving 
$F \in \Malg(\hatA(F) \hbt \hatA(\hat F))$. 

\fussy
\emph{Part 1/step 1:} 
\begin{align}
&(\hatA(F) \hbt \Kp(\L))F \nonumber\\
&= 
\left[ 
\scalebox{\scalefactor}{
\begin{tikzpicture}[baseline=-\the\dimexpr\fontdimen22\textfont2\relax+5, line width=.1em]
\draw (0,0) -- (2,0);		
\draw (0,-1) -- (0,1.5);
\draw[color=white,line width=.7em] (1,-1) -- (1,.5);
\draw (1,-1) -- (1,.5);
\draw (1,1) -- (1,1.5);
\draw (2,-.5) -- (2,.5);
\draw (0,-.5) -- (1,-.5);		
\ket{1}{1};
\bra{1}{.5};
\bra{2}{.5};
\ket{2}{-.5};
\end{tikzpicture}} 
\right]
= 
\left[ 
\scalebox{\scalefactor}{
\begin{tikzpicture}[baseline=-\the\dimexpr\fontdimen22\textfont2\relax+2.1cm, line width=.1em]
\draw (0,.25) -- (0,4);
\draw (1,.25) -- (1,.5);
\draw (1,1.5) -- (1,2.25);
\draw (1,3.75) -- (1,4);
\draw (2,1.5) -- (2,2.25);
\draw[color=white,line width=.7em] (.5,1.75) -- (1.5,1.75);
\draw (0,1.75) -- (2,1.75);
\draw (0,2.25) -- (1,2.25);
\braid[border height=0cm,height=1cm,width=1cm] at (1,1.5) s_1^{-1};
\braid[border height=0cm,height=1cm,width=1cm] at (1,3.25) s_1;
\ket{1}{3.75};
\bra{1}{3.25};
\bra{2}{3.25};
\ket{2}{.5};
\end{tikzpicture}} 
\right]
\stackrel{\eqref{eq:br-pentagon-1}}= 
\left[ 
\scalebox{\scalefactor}{
\begin{tikzpicture}[baseline=-\the\dimexpr\fontdimen22\textfont2\relax+2.1cm, line width=.1em]
\draw (0,.25) -- (0,4);
\draw (1,.25) -- (1,.5);
\draw (1,1.5) -- (1,2.25);
\draw (1,3.75) -- (1,4);
\draw (2,1.5) -- (2,2.25);
\draw (0,1.875) -- (1,1.875);
\draw (1,2.25) -- (2,2.25);
\draw[dashed] (1,1.5) -- (2,1.5);
\braid[border height=0cm,height=1cm,width=1cm] at (1,1.5) s_1^{-1};
\braid[border height=0cm,height=1cm,width=1cm] at (1,3.25) s_1;
\ket{1}{3.75};
\bra{1}{3.25};
\bra{2}{3.25};
\ket{2}{.5};
\end{tikzpicture}} 
\right]
\stackrel{\eqref{eq:remove-unitaries}}= 
\left[ 
\scalebox{\scalefactor}{
\begin{tikzpicture}[baseline=-\the\dimexpr\fontdimen22\textfont2\relax+1.8cm, line width=.1em]
\draw (0,.25) -- (0,3.5);
\draw (1,.25) -- (1,.5);
\draw (1,1.5) -- (1,2.5);
\draw (1,3) -- (1,3.5);
\draw (2,1.5) -- (2,2);
\draw (0,2) -- (1,2);
\draw[dashed] (1,1.5) -- (2,1.5);
\braid[border height=0cm,height=1cm,width=1cm] at (1,1.5) s_1^{-1};
\ket{1}{3};
\bra{1}{2.5};
\bra{2}{2};
\ket{2}{.5};
\end{tikzpicture}} 
\right]
\stackrel{\text{reg.}}= 
\left[ 
\scalebox{\scalefactor}{
\begin{tikzpicture}[baseline=-\the\dimexpr\fontdimen22\textfont2\relax+1.9cm, line width=.1em]
\draw (0,.5) -- (0,3.5);
\draw (1,.5) -- (1,1);
\draw (1,1.5) -- (1,2.5);
\draw (1,3) -- (1,3.5);
\draw (0,2) -- (1,2);
\ket{1}{3};
\bra{1}{2.5};
\ket{1}{1.5};
\bra{1}{1};
\end{tikzpicture}} 
\right]
\nonumber\\
&= \hatA(F) \ot_\mathrm{min} \Kp(\L) 
\stackrel[\ref{lem:on-braided-tensor-products}]{\text{Lem.}}= 
\hatA(F) \hbt \Kp(\L)
\end{align}
shows that $F \in \Malg(\hatA(F) \hbt \Kp(\L))$. 

\medskip

\emph{Part 1/step 2:} 
Applying the statement of step 1 to $\hat F$ gives 
$\hat F \in \Malg(\hatA(\hat F) \habt \Kp(\L))$; 
here, we have 
\enquote{$\habt$} instead of \enquote{$\hbt$}, 
since $\hat F$ satisfies the Pentagon equation 
in $\Cat^\mathrm{rev}$. 
Hence, with 
\begin{align}\label{eq:boxtimes--via--anti-bt}
\Kp(\L)) \hbt \hatA(\hat F)
&= [c_{L,L} \iota_2(\Kp(\L)) c_{L,L}^{-1} \iota_2(A(F))]
= c_{L,L} [\iota_2(\Kp(\L)) c_{L,L}^{-1} \iota_2(A(F)) c_{L,L}]c_{L,L}^{-1}
\nonumber\\
&= c_{L,L} (A(F) \habt \Kp(\L)) c_{L,L}^{-1} \ ,
\end{align}
we get 
\begin{align}
(\Kp(\L) \hbt \hatA(\hat F))F^* 
&= c_{L,L} (\hatA(\hat F) \habt \Kp(\L)) c_{L,L}^{-1} F^* 
= c_{L,L} (\hatA(\hat F) \habt \Kp(\L)) \hat F c_{L,L}^{-1} 
\nonumber\\
&= c_{L,L} (\hatA(\hat F) \habt \Kp(\L)) c_{L,L}^{-1} 
= \Kp(\L) \hbt \hatA(\hat F) \ ,
\end{align}
which shows $F \in \Malg(\Kp(\L) \hbt \hatA(\hat F))$.

\medskip

\emph{Part 1/step 3:} 
By the braided Pentagon equation and steps 1 \& 2, we have 
\begin{align}
\tilde F 
:= 
\scalebox{\scalefactor}{
\begin{tikzpicture}[baseline=-\the\dimexpr\fontdimen22\textfont2\relax, line width=.1em]
\draw (0,1) -- (0,-1);
\draw (1,1) -- (1,-1);
\draw (2,1) -- (2,-1);
\draw[color=white,line width=.7em] (.5,0) -- (1.5,0);
\draw (0,0) -- (2,0);
\node[anchor=east] at (0,0) {$F$};
\node[anchor=north] at (0,-1) {$L$};
\node[anchor=north] at (1,-1) {$L$};
\node[anchor=north] at (2,-1) {$L$};
\end{tikzpicture}} 
\ 
\stackrel{\eqref{eq:br-pentagon-1}}= 
\scalebox{\scalefactor}{
\begin{tikzpicture}[baseline=-\the\dimexpr\fontdimen22\textfont2\relax, line width=.1em]
\draw (0,1) -- (0,-1);
\draw (1,1) -- (1,-1);
\draw (2,1) -- (2,-1);
\draw[dashed] (1,-.6) -- (2,-.6);
\draw (0,-.2) -- (1,-.2);
\draw (1,.2) -- (2,.2);
\draw[dashed] (0,.6) -- (1,.6);
\node[anchor=east] at (0,.6) {$F^*$};
\node[anchor=west] at (2,.2) {$F$};
\node[anchor=east] at (0,-.2) {$F$};
\node[anchor=west] at (2,-.6) {$F^*$};
\node[anchor=north] at (0,-1) {$L$};
\node[anchor=north] at (1,-1) {$L$};
\node[anchor=north] at (2,-1) {$L$};
\end{tikzpicture}} 
\ \in \Malg(\hatA(F) \hbt \Kp(\L) \hbt \hatA(\hat F)) \ .
\end{align}
Put 
\begin{align}
S := 
\left[
\scalebox{\scalefactor}{
\begin{tikzpicture}[baseline=-\the\dimexpr\fontdimen22\textfont2\relax+.55cm, line width=.1em]
\draw (0,0) -- (0,1.25);
\draw (1,0) -- (1,.25);
\braid[border height=0cm,height=1cm,width=1cm] at (1,1.25) s_1^{-1};
\ket{2}{.25};
\end{tikzpicture}} 
\right]
\subseteq \B(\L \ot \L, \L\ot\L\ot\L) \ .
\end{align}
Then one easily checks that 
\begin{gather}
[S^*S] = \C\cdot \id_{\L\ot\L} \ , \qquad 
[S^*\tilde F S] = \C\cdot F \ , \nonumber\\ 
(\hatA(F) \hbt \Kp(\L) \hbt \hatA(\hat F)) \cdot S = S \cdot (\hatA(F) \hbt \hatA(\hat F)) \ .
\end{gather}
Hence, 
\begin{align}
&F \cdot (\hatA(F) \hbt \hatA(\hat F)) 
= [S^*\tilde F S \cdot (\hatA(F) \hbt \hatA(\hat F))] 
= [S^* \tilde F \cdot (\hatA(F) \hbt \Kp(\L) \hbt \hatA(\hat F)) S]
\nonumber\\
&\qquad\subseteq [S^* (\hatA(F) \hbt \Kp(\L) \hbt \hatA(\hat F)) S]
= [S^*S (\hatA(F) \hbt \hatA(\hat F))]
= \hatA(F) \hbt \hatA(\hat F) \ ,
\end{align}
proving $F \in \Malg(\hatA(F) \hbt \hatA(\hat F))$. 

\medskip

\emph{Part 2:} 
Again, we perform the proof in three steps.

\emph{Part 2/step 1:}
We have 
\begin{align}
&[(\hatA(F) \boxtimes \id_\L)F(\id_\L \boxtimes \Kp(\L))]
\nonumber\\
&= 
\left[
\scalebox{\scalefactor}{
\begin{tikzpicture}[baseline=-\the\dimexpr\fontdimen22\textfont2\relax+1.4cm, line width=.1em]
\draw (0,2) -- (2,2);
\draw (0,1.5) -- (1,1.5);
\draw (0,0) -- (0,2.75);
\draw (1,0) -- (1,.5);
\draw[color=white,line width=.7em] (1,1.75) -- (1,2.75);
\draw (1,1) -- (1,2.75);
\draw (2,1.5) -- (2,2.5);
\bra{2}{2.5};
\ket{2}{1.5};
\ket{1}{1};
\bra{1}{.5};
\end{tikzpicture}} 
\right]
= 
\left[
\scalebox{\scalefactor}{
\begin{tikzpicture}[baseline=-\the\dimexpr\fontdimen22\textfont2\relax+1.6cm, line width=.1em]
\draw (0,0) -- (0,3.25);
\draw (1,0) -- (1,.5);
\draw (1,1.75) -- (1,2.25);
\draw (1,3) -- (1,3.25);
\draw (2,1.75) -- (2,2.25);
\braid[border height=0cm,height=.75cm,width=1cm] at (1,3) s_1;
\braid[border height=0cm,height=.75cm,width=1cm] at (1,1.75) s_1^{-1};
\draw (0,2.25) -- (1,2.25);
\draw[color=white,line width=.7em] (.5,1.75) -- (1.5,1.75);
\draw (0,1.75) -- (2,1.75);
\bra{2}{3};
\ket{2}{1};
\ket{1}{1};
\bra{1}{.5};
\end{tikzpicture}} 
\right]
\stackrel{\eqref{eq:br-pentagon-1}}= 
\left[
\scalebox{\scalefactor}{
\begin{tikzpicture}[baseline=-\the\dimexpr\fontdimen22\textfont2\relax+1.6cm, line width=.1em]
\draw (0,0) -- (0,3.25);
\draw (1,0) -- (1,.5);
\draw (1,1.75) -- (1,2.25);
\draw (1,3) -- (1,3.25);
\draw (2,1.75) -- (2,2.25);
\draw (1,2.25) -- (2,2.25);
\draw (0,2) -- (1,2);
\draw[dashed] (1,1.75) -- (2,1.75);
\braid[border height=0cm,height=.75cm,width=1cm] at (1,3) s_1;
\braid[border height=0cm,height=.75cm,width=1cm] at (1,1.75) s_1^{-1};
\bra{2}{3};
\ket{2}{1};
\ket{1}{1};
\bra{1}{.5};
\end{tikzpicture}} 
\right]
\stackrel{\eqref{eq:remove-unitaries}}= 
\left[
\scalebox{\scalefactor}{
\begin{tikzpicture}[baseline=-\the\dimexpr\fontdimen22\textfont2\relax+1.5cm, line width=.1em]
\draw (0,0) -- (0,3);
\draw (1,0) -- (1,.5);
\draw (1,1) -- (1,1.75);
\draw (1,2.75) -- (1,3);
\draw (2,1.5) -- (2,1.75);
\draw (1,1.75) -- (2,1.75);
\draw (0,1.5) -- (1,1.5);
\braid[border height=0cm,height=1cm,width=1cm] at (1,2.75) s_1;
\bra{2}{2.75};
\ket{2}{1.5};
\ket{1}{1};
\bra{1}{.5};
\end{tikzpicture}} 
\right]
\stackrel[\text{of }F]{\text{reg.}}= 
\left[
\scalebox{\scalefactor}{
\begin{tikzpicture}[baseline=-\the\dimexpr\fontdimen22\textfont2\relax+1.5cm, line width=.1em]
\draw (0,0) -- (0,3);
\draw (1,0) -- (1,.5);
\draw (1,1) -- (1,2);
\draw (1,2.5) -- (1,3);
\draw (0,1.5) -- (1,1.5);
\ket{1}{2.5};
\bra{1}{2};
\ket{1}{1};
\bra{1}{.5};
\end{tikzpicture}} 
\right]
\nonumber\\
&= \hatA(F) \ot_\mathrm{min} \Kp(\L)
= \hatA(F) \hbt \Kp(\L) \ .
\end{align}

\medskip

\emph{Part 2/step 2:}
From 
$[(\hatA(\hat F) \habt \id_\L)\hat F(\id_\L \habt \Kp(\L))] 
	= \hatA(\hat F) \habt \Kp(\L)$ 
and \eqref{eq:boxtimes--via--anti-bt} we get 
\begin{align}
\Kp(\L) \hbt \hatA(\hat F) &= (\Kp(\L) \hbt \hatA(\hat F))^* 
= c_{L,L}(\hatA(\hat F) \habt \Kp(\L))^*c_{L,L}^{-1} 
\nonumber\\
&= c_{L,L}[(\hatA(\hat F) \habt \id_\L)\hat F(\id_\L \habt \Kp(\L))]^*c_{L,L}^{-1}
\nonumber\\
&= c_{L,L}[(\id_\L \habt \Kp(\L))\hat F^*(\hatA(\hat F) \habt \id_\L)]c_{L,L}^{-1}
\nonumber\\
&= [c_{L,L}(\id_\L \habt \Kp(\L))c_{L,L}^{-1} F c_{L,L}(A(F) \habt \id_\L)c_{L,L}^{-1}]
\nonumber\\
&= [(\Kp(\L) \hbt \id_\L) F (\id_\L \hbt \hatA(\hat F))]
\ .
\end{align}

\medskip

\emph{Part 2/step 3:}
Using the same technique as in the proof of 
part 3, it suffices to show that 
\begin{align}
[(\hatA(F) \hbt \Kp(\L) \hbt \id_\L)\tilde F(\id_\L \hbt \Kp(\L) \hbt \hatA(\hat F))] 
	= \hatA(F) \hbt \Kp(\L) \hbt \hatA(\hat F) \ .
\end{align}
This follows immediately from 
$\tilde F = F_{12}^*F_{23}F_{12}F_{23}^*$, 
together with steps 1 \& 2 of Part 1 and 
steps 1 \& 2 of Part 2.
\end{proof}

We are now prepared to turn $\hatA(F)$ into a 
bialgebra in $(\hatA(\Cat),\hbt)$, generalising 
\cite{Baaj-Skandalis:Unitaires}*{Prop.\,3.8}.

\begin{prop}
Let $\Cat$ be a regularly braided concrete monoidal $W^*$-category, 
let $L \in \Cat$ be an object with underlying 
Hilbert space $\L$, and let $F \in \U_{\Cat}(L \ot L)$ 
be a regular multiplicative unitary. 
For $a \in \hatA(F)$, put $\hat\Delta^\mathrm{op}(a) := F^*(\id_L \ot a)F$ 
and $\hat\Delta(a) := c_{L,L} \hat\Delta^\mathrm{op}(a) c_{L,L}^{-1}$. 
Then $(\hatA(F), \hat\Delta^\mathrm{op})$ and 
$(\hatA(F), \hat\Delta)$ are bi-simplifiable 
$C^*$-bialgebras in $(\hatA(\Cat),\habt)$, 
respectively, $(\hatA(\Cat),\hbt)$. 
\end{prop}
\begin{proof}
First, we need to show that the sets 
$(\hatA(F) \habt \id_L) \cdot \hat\Delta^\mathrm{op}(\hatA(F))$ 
and 
$(\id_L \habt \hatA(F)) \cdot \hat\Delta^\mathrm{op}(\hatA(F))$ 
are dense in $\hatA(F) \habt \hatA(F)$. 
We have 
\begin{align}
[(\hatA(F) \habt \id_L) \cdot \hat\Delta^\mathrm{op}(\hatA(F))] 
&= [c_{L,L}^{-1} \langle\L|_3 F_{23} |\L\rangle_3 c_{L,L} \cdot 
	\langle\L|_3 F_{12}^*F_{23}F_{12} |\L\rangle_3]
	\nonumber\\
&= [c_{L,L}^{-1} \langle\L|_3 F_{23} |\L\rangle_3 \cdot 
	\langle\L|_3 F_{23} (c_{L,L})_{12} F_{23} |\L\rangle_3]
	\nonumber\\
&= [c_{L,L}^{-1} \langle\L|_3 (\hatA(F) \ot_\mathrm{min} \Kp(\L))_{23} \cdot 
	F_{23} (c_{L,L})_{12} F_{23} |\L\rangle_3]
	\nonumber\\
&\stackrel[\ref{lem:on-braided-tensor-products}]{\text{Lem.}}= 
	[c_{L,L}^{-1} \langle\L|_3 (\hatA(F) \hbt \Kp(\L))_{23} \cdot 
	F_{23} (c_{L,L})_{12} F_{23} |\L\rangle_3]
	\nonumber\\
&\stackrel[\ref{thm:F-as-multiplier}]{\text{Thm.}}= 
	[c_{L,L}^{-1} \langle\L|_3 (\hatA(F) \hbt \Kp(\L))_{23} 
	(c_{L,L})_{12} F_{23} |\L\rangle_3]
	\nonumber\\
&= [c_{L,L}^{-1} \iota_2(\hatA(F)) c_{L,L} \cdot \hatA(F)]
= \hatA(F) \habt \hatA(F) \ ,
\end{align}
where we used standard leg notation, and put 
$|\L\rangle := \B(\C,\L)$ and $\langle\L| := \B(\L,\C)$. 
Concerning the denseness of the other set, we have 
\begingroup\allowdisplaybreaks
\begin{align}
&[(\id_L \habt \hatA(F)) \cdot \hat\Delta^\mathrm{op}(\hatA(F))]
\nonumber\\*
&= 
\left[
\scalebox{\scalefactor}{
\begin{tikzpicture}[baseline=-\the\dimexpr\fontdimen22\textfont2\relax+1.5cm, line width=.1em]
\draw (0,0) -- (0,3);
\draw (1,0) -- (1,3);
\draw (2,.25) -- (2,1.25);
\draw (2,1.75) -- (2,2.75);
\draw (0,.5) -- (1,.5);
\draw (1,.75) -- (2,.75);
\draw[dashed] (0,1) -- (1,1);
\draw (1,2.25) -- (2,2.25);
\bra{2}{2.75};
\ket{2}{1.75};
\bra{2}{1.25};
\ket{2}{.25};
\end{tikzpicture}} 
\right]
\stackrel{\eqref{eq:br-pentagon-1}}= 
\left[
\scalebox{\scalefactor}{
\begin{tikzpicture}[baseline=-\the\dimexpr\fontdimen22\textfont2\relax+1.5cm, line width=.1em]
\draw (0,0) -- (0,3);
\draw (1,0) -- (1,3);
\draw (2,.25) -- (2,1.25);
\draw (2,1.75) -- (2,2.75);
\draw[color=white,line width=.7em] (.5,1) -- (1.5,1);
\draw (0,1) -- (2,1);
\draw (1,.5) -- (2,.5);
\draw (1,2.25) -- (2,2.25);
\bra{2}{2.75};
\ket{2}{1.75};
\bra{2}{1.25};
\ket{2}{.25};
\end{tikzpicture}} 
\right]
\stackrel[\text{braiding}]{\text{reg. of}}= 
\left[
\scalebox{\scalefactor}{
\begin{tikzpicture}[baseline=-\the\dimexpr\fontdimen22\textfont2\relax+1.5cm, line width=.1em]
\draw (0,0) -- (0,3);
\draw (1,0) -- (1,3);
\draw (2,.25) -- (2,1.25);
\draw (3,.25) -- (3,1.25);
\braid[border height=.5cm,height=1cm,width=1cm] at (2,2.75) s_1;
\draw[color=white,line width=.7em] (1.5,.75) -- (2.5,.75);
\draw (1,.75) -- (3,.75);
\draw (1,1.25) -- (2,1.25);
\draw[color=white,line width=.7em] (.5,2.25) -- (1.5,2.25);
\draw (0,2.25) -- (2,2.25);
\bra{2}{2.75};
\bra{3}{2.75};
\ket{2}{.25};
\ket{3}{.25};
\end{tikzpicture}} 
\right]
\stackrel{\eqref{eq:br-pentagon-1}}= 
\left[
\scalebox{\scalefactor}{
\begin{tikzpicture}[baseline=-\the\dimexpr\fontdimen22\textfont2\relax+1.5cm, line width=.1em]
\draw (0,0) -- (0,3);
\draw (1,0) -- (1,3);
\draw (2,.25) -- (2,1.25);
\draw (3,.25) -- (3,1.25);
\braid[border height=.5cm,height=1cm,width=1cm] at (2,2.75) s_1;
\draw (1,1) -- (2,1);
\draw[dashed] (2,.75) -- (3,.75);
\draw (2,1.25) -- (3,1.25);
\draw[color=white,line width=.7em] (.5,2.25) -- (1.5,2.25);
\draw (0,2.25) -- (2,2.25);
\bra{2}{2.75};
\bra{3}{2.75};
\ket{2}{.25};
\ket{3}{.25};
\end{tikzpicture}} 
\right]
\nonumber\\
&\stackrel{\eqref{eq:remove-unitaries}}= 
\left[
\scalebox{\scalefactor}{
\begin{tikzpicture}[baseline=-\the\dimexpr\fontdimen22\textfont2\relax+1.5cm, line width=.1em]
\draw (0,0) -- (0,3);
\draw (1,0) -- (1,3);
\draw (2,.25) -- (2,1.25);
\draw (2,2.25) -- (2,2.75);
\draw (3,.75) -- (3,1.25);
\braid[border height=.0cm,height=1cm,width=1cm] at (2,2.25) s_1;
\draw (1,.75) -- (2,.75);
\draw (2,1.25) -- (3,1.25);
\draw[color=white,line width=.7em] (.5,2.25) -- (1.5,2.25);
\draw (0,2.25) -- (2,2.25);
\bra{2}{2.75};
\bra{3}{2.25};
\ket{2}{.25};
\ket{3}{.75};
\end{tikzpicture}} 
\right]
\stackrel[\text{of }F]{\text{reg.}}= 
\left[
\scalebox{\scalefactor}{
\begin{tikzpicture}[baseline=-\the\dimexpr\fontdimen22\textfont2\relax+1.5cm, line width=.1em]
\draw (0,0) -- (0,3);
\draw (1,0) -- (1,3);
\draw (2,.25) -- (2,1.25);
\draw (2,1.75) -- (2,2.75);
\draw (1,.75) -- (2,.75);
\draw[color=white,line width=.7em] (.5,2.25) -- (1.5,2.25);
\draw (0,2.25) -- (2,2.25);
\bra{2}{2.75};
\ket{2}{1.75};
\bra{2}{1.25};
\ket{2}{.25};
\end{tikzpicture}} 
\right]
= \hatA(F) \habt \hatA(F) \ .
\end{align}
\endgroup
Clearly, $\hat\Delta^\mathrm{op}$ is of the form 
\ref{eq:morphisms-in-hatA(C)}. That it is a 
non-degenerate $^*$-homomorphism is immediate from 
the Podle{\'s} conditions. 
Hence, $(\hatA(F),\hat\Delta^\mathrm{op})$ 
is a bi-simplifiable $C^*$-bialgebra in $(\hatA(\Cat),\habt)$.
The statement about $(\hatA(F),\hat\Delta)$ 
follows using the identity 
$\hatA(F) \hbt \hatA(F) = c_{L,L} (\hatA(F) \habt \hatA(F)) c_{L,L}^{-1}$.
\end{proof}
\fussy

\section{Yetter-Drinfeld categories of regular braided multiplicative unitaries}
\label{sec:YD-categories}

Our starting point in the last section was a 
regularly braided concrete monoidal $W^*$-category. 
We had seen, that the category $\Hilb$ with the 
tensor flip $\Sigma$ is an example of such a category. 
The purpose of the present section is to provide 
more examples: 
We know already that the category of Yetter-Drinfeld 
modules over a regular multiplicative unitary in a regularly braided 
concrete monoidal $W^*$-category 
is unitarily braided (Remark \ref{rmk:regular-MU}). 
Here, we will see that its braiding is also regular.

\begin{lem}\label{lem:regularity-of-coreps-and-reps}
Let $\Cat$ be a \altern{semi-regularly}{regularly} braided concrete monoidal $W^*$-category, 
let $L \in \Cat$ be an object with underlying 
Hilbert space $\L$, and let $F \in \U_{\Cat}(L \ot L)$ 
be a \altern{semi-regular}{regular} 
multiplicative unitary. 
We have: 
\begin{enumerate}
\item 
If $(H,U) \in \Cat^F$, then 
$\Kp(\H,\L) \inceq 
	[(\id\ot\omega)(c_{L,H}^{-1} U)\,|\,\omega\in\B(\L,\H)_*]$.
\item 
If $(K,V) \in {}_F\Cat$, then 
$\Kp(\L,\H) \inceq 
	[(\id\ot\omega)(c_{K,L}^{-1} V)\,|\,\omega\in\B(\H,\L)_*]$.
\end{enumerate}
\end{lem}
\begin{proof}
\emph{Part 1:}
\begingroup\allowdisplaybreaks
\begin{align}
&[(\id\ot\omega)(c_{L,H}^{-1} U)\,|\,\omega\in\B(\L,\H)_*] \nonumber\\*
&= 
\left[
\scalebox{\scalefactor}{
\begin{tikzpicture}[baseline=-\the\dimexpr\fontdimen22\textfont2\relax, line width=.1em]
\draw (0,-1.6) -- (0,.6)
	.. controls (0,1.2) and (1,1.2) .. (1,1.8);
\draw (1,-.6) -- (1,.6) 
	.. controls (1,1.1) and (2,1.1) .. (2,1.6);
\draw[color=white,line width=.7em] (2,-.6) -- (2,.6) 
	.. controls (2,1.35) and (0,1.35) .. (0,2.1);
\draw (2,-.6) -- (2,.6) 
	.. controls (2,1.35) and (0,1.35) .. (0,2.1);
\braid[border height=0cm,height=.75cm,width=1cm] at (1,-.6) s_1^{-1};
\draw[color=white,line width=.7em] (.5,0) -- (1.5,0);
\draw (0,0) -- (2,0);
\bra{1}{1.9};
\bra{2}{1.7};
\ket{1}{-1.45};
\ket{2}{-1.45};
\node at (-.25,0) {$U$};
\node at (0,-1.85) {$H$};
\node at (1,-1.85) {$L$};
\node at (2,-1.85) {$L$};
\end{tikzpicture}} 
\right]
\stackrel{\eqref{eq:br-pentagon-2}}= 
\left[
\scalebox{\scalefactor}{
\begin{tikzpicture}[baseline=-\the\dimexpr\fontdimen22\textfont2\relax, line width=.1em]
\draw (0,-1.6) -- (0,.6)
	.. controls (0,1.2) and (1,1.2) .. (1,1.8);
\draw (1,-.6) -- (1,.6) 
	.. controls (1,1.1) and (2,1.1) .. (2,1.6);
\draw[color=white,line width=.7em] (2,-.6) -- (2,.6) 
	.. controls (2,1.35) and (0,1.35) .. (0,2.1);
\draw (2,-.6) -- (2,.6) 
	.. controls (2,1.35) and (0,1.35) .. (0,2.1);
\braid[border height=0cm,height=.75cm,width=1cm] at (1,-.6) s_1^{-1};
\draw[dashed] (1,-.6) -- (2,-.6);
\draw (0,-.2) -- (1,-.2);
\draw (1,.2) -- (2,.2);
\draw[dashed] (0,.6) -- (1,.6);
\bra{1}{1.9};
\bra{2}{1.7};
\ket{1}{-1.45};
\ket{2}{-1.45};
\node at (-.3,.6) {$U^*$};
\node at (-.25,-.2) {$U$};
\node at (2.25,.2) {$F$};
\node at (2.3,-.6) {$F^*$};
\node at (0,-1.85) {$H$};
\node at (1,-1.85) {$L$};
\node at (2,-1.85) {$L$};
\end{tikzpicture}} 
\right]
\stackrel{\eqref{eq:remove-unitaries}}= 
\left[
\scalebox{\scalefactor}{
\begin{tikzpicture}[baseline=-\the\dimexpr\fontdimen22\textfont2\relax+.3cm, line width=.1em]
\draw (0,-.9) -- (0,.4)
	.. controls (0,.9) and (1,.95) .. (1,1.5);
\draw (1,-.6) -- (1,.2) 
	.. controls (1,.7) and (2,.7) .. (2,1.2);
\draw[color=white,line width=.7em] (2,-.1) -- (2,.2) 
	.. controls (2,.9) and (0,.9) .. (0,2);
\draw (2,-.1) -- (2,.2) 
	.. controls (2,.9) and (0,.9) .. (0,2);
\draw (0,-.2) -- (1,-.2);
\draw (1,.2) -- (2,.2);
\bra{1}{1.6};
\bra{2}{1.3};
\ket{1}{-.7};
\ket{2}{-.2};
\node at (-.25,-.2) {$U$};
\node at (2.25,.2) {$F$};
\node at (0,-1.15) {$H$};
\node at (1,-1.15) {$L$};
\node at (2,-.65) {$L$};
\end{tikzpicture}} 
\right]
\stackrel[\text{of }F]{\text{(semi-)reg.}}\rinceq 
\left[
\scalebox{\scalefactor}{
\begin{tikzpicture}[baseline=-\the\dimexpr\fontdimen22\textfont2\relax+.3cm, line width=.1em]
\draw (0,-.9) -- (0,.4)
	.. controls (0,1.2) and (1,1.2) .. (1,1.5);
\draw (1,-.6) -- (1,.2);
\draw[color=white,line width=.7em] (1,.8) .. controls (1,1.3) and (0,1.3) .. (0,1.9);
\draw (1,.8) .. controls (1,1.3) and (0,1.3) .. (0,1.9);
\draw (0,-.2) -- (1,-.2);
\bra{1}{1.6};
\ket{1}{.8};
\bra{1}{.3};
\ket{1}{-.7};
\node at (-.25,-.2) {$U$};
\node at (0,-1.15) {$H$};
\node at (1,-1.15) {$L$};
\end{tikzpicture}} 
\right]
\nonumber\\
&\stackrel[\text{braiding}]{\text{(semi-)reg. of}}\rinceq 
\left[
\scalebox{\scalefactor}{
\begin{tikzpicture}[baseline=-\the\dimexpr\fontdimen22\textfont2\relax, line width=.1em]
\draw (0,-.9) -- (0,.2);
\draw (1,-.6) -- (1,.2);
\draw (0,.7) -- (0,1.3);
\draw (0,-.2) -- (1,-.2);
\ket{0}{.8};
\bra{0}{.3};
\bra{1}{.3};
\ket{1}{-.7};
\node at (-.25,-.2) {$U$};
\node at (0,-1.15) {$H$};
\node at (1,-1.15) {$L$};
\end{tikzpicture}} 
\right]
= \Kp(\H,\L) \ ,
\end{align}
\endgroup
where in the last step we made use of the non-degeneracy of the algebra 
$[(\id\ot\omega)(U)\,|\,\omega\in\B(\L)_*]$, 
which one shows in the same way as 
Proposition \ref{prop:F-slices-are-non-deg-algebras}. 

\medskip

\emph{Part 2:} 
As $\hat F$ is a \altern{semi-regular}{regular} multiplicative unitary in 
$\Cat^\mathrm{rev}$ and as, with $\hat V := c_{K,L}^{-1}V^*c_{K,L}$, 
we have $(K,\hat V) \in (\Cat^\mathrm{rev})^{\hat F}$, 
the statement of part 1 implies 
\begin{align}
\Kp(\K,\L) 
&\inceq [(\id\ot\omega)(c_{K,L}\hat V)\,|\,\omega \in\B(\L,\K)_*] 
= [(\id\ot\omega)(V^*c_{K,L})\,|\,\omega \in\B(\L,\K)_*] 
\nonumber\\
&\ \ = [(\id\ot\omega)(c_{K,L}^{-1}V)\,|\,\omega \in\B(\K,\L)_*]^* \ ,
\end{align}
showing the claim.
\end{proof}

\begin{prop}\label{prop:YD-cats-of-good-regular-MUs-have-regular-braiding}
Let $\Cat$ be a \altern{semi-regularly}{regularly} braided 
concrete monoidal $W^*$-category, 
let $L \in \Cat$ be an object with underlying 
Hilbert space $\L$, and let $F \in \U_{\Cat}(L \ot L)$ 
be a \altern{semi-regular}{regular} multiplicative unitary. 
Then the braiding $\Phi$ on the category $\FYDF$ is 
\altern{semi-regular}{regular}\ . 
\end{prop}
\begin{proof}
The proof is similar to that of statement 1 in 
Lemma \ref{lem:regularity-of-coreps-and-reps}, but 
employing statement 2 of 
Lemma \ref{lem:regularity-of-coreps-and-reps} 
instead of \altern{semi-regularity}{regularity} of $F$: 
If $(H,U_H,V_H), (K,U_K,V_K) \in \FYDF$, then  
\begingroup\allowdisplaybreaks
\begin{align}
&[(\id\ot\omega)(\Phi_{K,H})\,|\,\omega\in\B(\K,\H)_*] \nonumber\\*
&= 
\left[
\scalebox{\scalefactor}{
\begin{tikzpicture}[baseline=-\the\dimexpr\fontdimen22\textfont2\relax, line width=.1em]
\draw (0,-1.6) -- (0,.6)
	.. controls (0,1.2) and (1,1.2) .. (1,1.8);
\draw (1,-.6) -- (1,.6) 
	.. controls (1,1.1) and (2,1.1) .. (2,1.6);
\draw[color=white,line width=.7em] (2,-.6) -- (2,.6) 
	.. controls (2,1.35) and (0,1.35) .. (0,2.1);
\draw (2,-.6) -- (2,.6) 
	.. controls (2,1.35) and (0,1.35) .. (0,2.1);
\braid[border height=0cm,height=.75cm,width=1cm] at (1,-.6) s_1^{-1};
\draw[color=white,line width=.7em] (.5,0) -- (1.5,0);
\draw (0,0) -- (2,0);
\bra{1}{1.9};
\bra{2}{1.7};
\ket{1}{-1.45};
\ket{2}{-1.45};
\node at (-.7,0) {$V_K\ast U_H$};
\node at (0,-1.85) {$H$};
\node at (1,-1.85) {$K$};
\node at (2,-1.85) {$L$};
\end{tikzpicture}} 
\right]
= 
\left[
\scalebox{\scalefactor}{
\begin{tikzpicture}[baseline=-\the\dimexpr\fontdimen22\textfont2\relax, line width=.1em]
\draw (0,-1.6) -- (0,.6)
	.. controls (0,1.2) and (1,1.2) .. (1,1.8);
\draw (1,-.6) -- (1,.6) 
	.. controls (1,1.1) and (2,1.1) .. (2,1.6);
\draw[color=white,line width=.7em] (2,-.6) -- (2,.6) 
	.. controls (2,1.35) and (0,1.35) .. (0,2.1);
\draw (2,-.6) -- (2,.6) 
	.. controls (2,1.35) and (0,1.35) .. (0,2.1);
\braid[border height=0cm,height=.75cm,width=1cm] at (1,-.6) s_1^{-1};
\draw[dashed] (1,-.6) -- (2,-.6);
\draw (0,-.2) -- (1,-.2);
\draw (1,.2) -- (2,.2);
\draw[dashed] (0,.6) -- (1,.6);
\bra{1}{1.9};
\bra{2}{1.7};
\ket{1}{-1.45};
\ket{2}{-1.45};
\node at (-.3,.6) {$U_H^*$};
\node at (-.25,-.2) {$U_H$};
\node at (2.25,.2) {$V_K$};
\node at (2.3,-.6) {$V_K^*$};
\node at (0,-1.85) {$H$};
\node at (1,-1.85) {$K$};
\node at (2,-1.85) {$L$};
\end{tikzpicture}} 
\right]
\stackrel{\eqref{eq:remove-unitaries}}= 
\left[
\scalebox{\scalefactor}{
\begin{tikzpicture}[baseline=-\the\dimexpr\fontdimen22\textfont2\relax+.3cm, line width=.1em]
\draw (0,-.9) -- (0,.4)
	.. controls (0,.9) and (1,.95) .. (1,1.5);
\draw (1,-.6) -- (1,.2) 
	.. controls (1,.7) and (2,.7) .. (2,1.2);
\draw[color=white,line width=.7em] (2,-.1) -- (2,.2) 
	.. controls (2,.9) and (0,.9) .. (0,2);
\draw (2,-.1) -- (2,.2) 
	.. controls (2,.9) and (0,.9) .. (0,2);
\draw (0,-.2) -- (1,-.2);
\draw (1,.2) -- (2,.2);
\bra{1}{1.6};
\bra{2}{1.3};
\ket{1}{-.7};
\ket{2}{-.2};
\node at (-.25,-.2) {$U_H$};
\node at (2.25,.2) {$V_K$};
\node at (0,-1.15) {$H$};
\node at (1,-1.15) {$L$};
\node at (2,-.65) {$K$};
\end{tikzpicture}} 
\right]
\stackrel[\ref{lem:regularity-of-coreps-and-reps}]{\text{Lem.}}\rinceq 
\left[
\scalebox{\scalefactor}{
\begin{tikzpicture}[baseline=-\the\dimexpr\fontdimen22\textfont2\relax+.3cm, line width=.1em]
\draw (0,-.9) -- (0,.4)
	.. controls (0,1.2) and (1,1.2) .. (1,1.5);
\draw (1,-.6) -- (1,.2);
\draw[color=white,line width=.7em] (1,.8) .. controls (1,1.3) and (0,1.3) .. (0,1.9);
\draw (1,.8) .. controls (1,1.3) and (0,1.3) .. (0,1.9);
\draw (0,-.2) -- (1,-.2);
\bra{1}{1.6};
\ket{1}{.8};
\bra{1}{.3};
\ket{1}{-.7};
\node at (.5,-.5) {$U_H$};
\node at (0,-1.15) {$H$};
\node at (1,-1.15) {$L$};
\node at (.6,.7) {$K$};
\end{tikzpicture}} 
\right]
\nonumber\\
&\stackrel[\text{braiding}]{\text{(semi-)reg. of}}\rinceq 
\left[
\scalebox{\scalefactor}{
\begin{tikzpicture}[baseline=-\the\dimexpr\fontdimen22\textfont2\relax, line width=.1em]
\draw (0,-.9) -- (0,.2);
\draw (1,-.6) -- (1,.2);
\draw (0,.7) -- (0,1.3);
\draw (0,-.2) -- (1,-.2);
\ket{0}{.8};
\bra{0}{.3};
\bra{1}{.3};
\ket{1}{-.7};
\node at (-.25,-.2) {$U_H$};
\node at (0,-1.15) {$H$};
\node at (1,-1.15) {$L$};
\node at (-.25,1.15) {$K$};
\end{tikzpicture}} 
\right]
= \Kp(\H,\K) \ .
\end{align}
\endgroup
\end{proof}

\section{Semi-direct products of regular braided multiplicative unitaries}
\label{sec:semi-direct-products}

Let $\Cat$ be a concrete regularly braided $W^*$-category, 
$K$ an object in $\Cat$, let 
$W \in \U_{\Cat}(K \ot K)$ be a regular multiplicative unitary 
in $\Cat$. Moreover, let $L = (\chL, U, V)$ be an object in $\WYDW$ and 
$F \in \U_{\WYDW}(L \ot L)$ be a multiplicative unitary in 
$\WYDW$. 
Define the \emph{semi-direct product multiplicative unitary} 
$W \ltimes F$ by 
\begin{align}
W \ltimes F := 
\scalebox{\scalefactor}{
\begin{tikzpicture}[baseline=-\the\dimexpr\fontdimen22\textfont2\relax+1.05cm, line width=.1em]
\draw[color=gray] (0,0) -- (0,2.15);
\draw (.5,0) -- (.5,2.15);
\draw[color=gray] (1.5,0) -- (1.5,.25) 
	.. controls (1.5,.5) and (2.4,.5) .. (2.4,.75)
	.. controls (2.4,1) and (1.5,1) .. (1.5,1.25)
	-- (1.5,2.15);
\draw[color=white,line width=.7em] (2,0) -- (2,2.15);
\draw (.5,.75) -- (2,.75);
\draw[color=gray] (1.5,.25) -- (2,.25);
\draw[color=gray,dashed] (1.5,1.25) -- (2,1.25);
\draw[color=gray] (.5,1.5) -- (1.5,1.5);
\draw[color=white, line width=.7em] (.2,1.9) -- (1,1.9);
\draw[color=gray] (0,1.9) -- (1.5,1.9);
\draw (2,0) -- (2,2.15);
\node at (-.23,1.9) {$W$};
\node at (.3,1.5) {$U$};
\node at (2.28,1.25) {$V^*$};
\node at (.3,.75) {$F$};
\node at (2.2,.25) {$V$};
\node at (0,-.25) {$H$};
\node at (.5,-.25) {$\chL$};
\node at (1.5,-.25) {$H$};
\node at (2,-.25) {$\chL$};
\end{tikzpicture}} 
\in \U_{\Cat}(K \ot \chL \ot K \ot \chL) \ . 
\end{align}
That $W \ltimes F$ is indeed a multiplicative unitary 
in $\Cat$ is shown, for the special case 
$(\Cat,\ot,c) = (\Hilb,\ot,\Sigma)$, 
in \cite{Roy:Qgrp_with_proj}*{Thm.\,6.15}; 
details for the case considered here 
will be presented in a separate paper. 

Before investigating the regularity properties of 
semi-direct products, we introduce the notion of 
fixed vectors.

\begin{df}\label{df:fixed-vectors}
Let $W \in \U_{\Cat}(K \ot K)$ be a multiplicative unitary. 
A \emph{fixed vector} for $W$ is a morphism 
$e \in \Hom_\Cat(\one,K)$ such that 
$W(e \ot \xi) = e \ot \xi$ holds for all $\xi \in \K$. 
\end{df}

\begin{rmk}\label{rmk:fixed-vectors}
\begin{enumerate}
\item 
An ordinary multiplicative unitary $W$ is said to be 
of \emph{compact type}, if $A(W)$ is unital. 
It is shown in \cite{Baaj-Skandalis:Unitaires}*{Prop.\,1.10} for 
multiplicative unitaries $W$ on separable Hilbert spaces, that 
$W$ is of compact type if and only if it admits a non-zero fixed vector.
\item 
If $e$ is a fixed vector for $W$, then we also have 
\begin{align}
\scalebox{\scalefactor}{
\begin{tikzpicture}[baseline=-\the\dimexpr\fontdimen22\textfont2\relax, line width=.1em]
\draw (0,.6) -- (0,-.35);
\draw (1,.6) -- (1,-.6);
\draw[dashed] (0,.2) -- (1,.2);
\node[draw,fill=white,inner sep=.16em] at (0,-.35) {$e$};
\node[anchor=east] at (0,.2) {$W^*$};
\node[anchor=north] at (1,-.6) {$K$};
\node[anchor=south] at (0,.6) {$K$};
\node[anchor=south] at (1,.6) {$K$};
\end{tikzpicture}} 
= 
\scalebox{\scalefactor}{
\begin{tikzpicture}[baseline=-\the\dimexpr\fontdimen22\textfont2\relax, line width=.1em]
\draw (0,.6) -- (0,-.35);
\draw (1,.6) -- (1,-.6);
\node[draw,fill=white,inner sep=.16em] at (0,-.35) {$e$};
\node[anchor=north] at (1,-.6) {$K$};
\node[anchor=south] at (0,.6) {$K$};
\node[anchor=south] at (1,.6) {$K$};
\end{tikzpicture}} 
\qquad
\text{and}
\qquad
\scalebox{\scalefactor}{
\begin{tikzpicture}[baseline=-\the\dimexpr\fontdimen22\textfont2\relax, line width=.1em]
\draw (0,.5) -- (0,-.6);
\draw (1,.6) -- (1,-.6);
\draw (0,-.2) -- (1,-.2);
\node[draw,fill=white,inner sep=.1em] at (0,.35) {$e^*$};
\node[anchor=west] at (1,-.2) {$W$};
\node[anchor=north] at (0,-.6) {$K$};
\node[anchor=north] at (1,-.6) {$K$};
\node[anchor=south] at (1,.6) {$K$};
\end{tikzpicture}} 
= 
\scalebox{\scalefactor}{
\begin{tikzpicture}[baseline=-\the\dimexpr\fontdimen22\textfont2\relax, line width=.1em]
\draw (0,.5) -- (0,-.6);
\draw (1,.6) -- (1,-.6);
\node[draw,fill=white,inner sep=.1em] at (0,.35) {$e^*$};
\node[anchor=north] at (0,-.6) {$K$};
\node[anchor=north] at (1,-.6) {$K$};
\node[anchor=south] at (1,.6) {$K$};
\end{tikzpicture}} 
\ .
\end{align}
\item
If $e$ is a non-zero fixed vector for $W$, and $(H,V) \in {}_W\Cat$, 
then we have 
\begin{align}
\scalebox{\scalefactor}{
\begin{tikzpicture}[baseline=-\the\dimexpr\fontdimen22\textfont2\relax, line width=.1em]
\draw (0,.6) -- (0,-.35);
\draw (1,.6) -- (1,-.6);
\draw (0,.2) -- (1,.2);
\node[draw,fill=white,inner sep=.16em] at (0,-.35) {$e$};
\node[anchor=east] at (0,.2) {$V$};
\node[anchor=north] at (1,-.6) {$H$};
\node[anchor=south] at (0,.6) {$K$};
\node[anchor=south] at (1,.6) {$H$};
\end{tikzpicture}} 
= 
\frac 1{\|e\|^2} \cdot 
\scalebox{\scalefactor}{
\begin{tikzpicture}[baseline=-\the\dimexpr\fontdimen22\textfont2\relax, line width=.1em]
\draw (0,1) -- (0,-1);
\draw (1,1) -- (1,-1);
\draw (2,1) -- (2,-1);
\draw[color=white,line width=.7em] (.5,0) -- (1.5,0);
\draw (0,0) -- (2,0);
\node[draw,fill=white,inner sep=.15em] at (0,-.9) {$e$};
\node[draw,fill=white,inner sep=.15em] at (1,-.9) {$e$};
\node[draw,fill=white,inner sep=.1em] at (1,.9) {$e^*$};
\node[anchor=east] at (0,0) {$V$};
\node[anchor=north] at (2,-1) {$H$};
\end{tikzpicture}} 
= 
\frac 1{\|e\|^2} \cdot 
\scalebox{\scalefactor}{
\begin{tikzpicture}[baseline=-\the\dimexpr\fontdimen22\textfont2\relax, line width=.1em]
\draw (0,1) -- (0,-1);
\draw (1,1) -- (1,-1);
\draw (2,1) -- (2,-1);
\draw[dashed] (1,-.6) -- (2,-.6);
\draw (0,-.2) -- (1,-.2);
\draw (1,.2) -- (2,.2);
\draw[dashed] (0,.6) -- (1,.6);
\node[draw,fill=white,inner sep=.15em] at (0,-.9) {$e$};
\node[draw,fill=white,inner sep=.15em] at (1,-.9) {$e$};
\node[draw,fill=white,inner sep=.1em] at (1,.9) {$e^*$};
\node[anchor=east] at (0,.6) {$W^{*}$};
\node[anchor=east] at (0,-.2) {$W$};
\node[anchor=west] at (2,.2) {$V$};
\node[anchor=west] at (2,-.6) {$V^{*}$};
\node[anchor=north] at (2,-1) {$H$};
\end{tikzpicture}} 
= 
\scalebox{\scalefactor}{
\begin{tikzpicture}[baseline=-\the\dimexpr\fontdimen22\textfont2\relax, line width=.1em]
\draw (0,.6) -- (0,-.35);
\draw (1,.6) -- (1,-.6);
\node[draw,fill=white,inner sep=.16em] at (0,-.35) {$e$};
\node[anchor=north] at (1,-.6) {$H$};
\node[anchor=south] at (0,.6) {$K$};
\node[anchor=south] at (1,.6) {$H$};
\end{tikzpicture}} 
\ .
\end{align}
\end{enumerate}
\end{rmk}

\begin{prop}
Let $\Cat$ be a concrete $W^*$-category with regular braiding, 
let $W \in \U_\Cat(K \ot K)$ be a regular multiplicative unitary, 
and $F \in \U_{\WYDW}(L \ot L)$ 
be a multiplicative unitary. 
\begin{enumerate}
\item 
If $F$ is regular, then so is $W \ltimes F$. 
\item 
If $W$ admits a non-zero fixed vector and $W \ltimes F$ is regular, 
then $F$ is regular.
\end{enumerate}
\end{prop}
\begin{proof}
\emph{Part 1:} 
Let $L = (\chL, U, V) \in \WYDW$ 
and write $Z := (V \ast U)^*$, where the product 
$\ast$ is as in Section \ref{ssec:br-mult-unitaries}. 
We have 
\begingroup\allowdisplaybreaks
\begin{align}\label{eq:prop-regularity-of-semidirect-product--1}
&\left[
\scalebox{\scalefactor}{
\begin{tikzpicture}[baseline=-\the\dimexpr\fontdimen22\textfont2\relax+1.3cm, line width=.1em]
\draw[color=gray] (0,-.25) -- (0,2.15)				
	.. controls (0,2.6) and (1.5,2.6) .. (1.5,3);
\draw (.5,-.25) -- (.5,2.1)							
	.. controls (.5,2.5) and (2,2.5) .. (2,2.9);
\draw[color=white, line width=.7em] (1.5,2)		
	.. controls (1.5,2.4) and (0,2.4) .. (0,3);
\draw[color=gray] (1.5,0) -- (1.5,.25) 				
	.. controls (1.5,.5) and (2.4,.5) .. (2.4,.75)
	.. controls (2.4,1) and (1.5,1) .. (1.5,1.25)
	-- (1.5,2)
	.. controls (1.5,2.4) and (0,2.4) .. (0,3)
	-- (0,3.25);
\draw[color=white,line width=.7em] (2,0) -- (2,2.15)		
	.. controls (2,2.65) and (.5,2.65) .. (.5,3.1)
	-- (.5,3.25);
\draw (.5,.75) -- (2,.75);							
\draw[color=gray] (1.5,.25) -- (2,.25);
\draw[color=gray,dashed] (1.5,1.25) -- (2,1.25);
\draw[color=gray] (.5,1.5) -- (1.5,1.5);			
\draw[color=white, line width=.7em] (.2,1.9) -- (1,1.9);
\draw[color=gray] (0,1.9) -- (1.5,1.9);
\draw (2,0) -- (2,2.15)
	.. controls (2,2.65) and (.5,2.65) .. (.5,3.1)
	-- (.5,3.25);
\bra{1.5}{3.1};
\bra{2}{3}
\ket{1.5}{-.1};
\ket{2}{-.1}
\node at (-.23,1.9) {$W$};
\node at (.3,1.5) {$U$};
\node at (2.28,1.25) {$V^*$};
\node at (.3,.75) {$F$};
\node at (2.2,.25) {$V$};
\node at (0,-.5) {$H$};
\node at (.5,-.5) {$\chL$};
\node at (1.5,-.5) {$H$};
\node at (2,-.5) {$\chL$};
\end{tikzpicture}} 
\right]
\stackrel{\eqref{eq:remove-unitaries}}= 
\left[
\scalebox{\scalefactor}{
\begin{tikzpicture}[baseline=-\the\dimexpr\fontdimen22\textfont2\relax+1.3cm, line width=.1em]
\draw[color=gray] (0,-.25) -- (0,2.15)				
	.. controls (0,2.6) and (1.5,2.6) .. (1.5,3);
\draw (.5,-.25) -- (.5,2.1)							
	.. controls (.5,2.5) and (2,2.5) .. (2,2.9);
\draw[color=white, line width=.7em] (1.5,2)		
	.. controls (1.5,2.4) and (0,2.4) .. (0,3);
\draw[color=gray] (2.5,.5) -- (2.5,.75) 				
	.. controls (2.5,1) and (1.5,1) .. (1.5,1.25)
	-- (1.5,2)
	.. controls (1.5,2.4) and (0,2.4) .. (0,3)
	-- (0,3.25);
\draw[color=white,line width=.7em] (2,0) -- (2,2.15)		
	.. controls (2,2.65) and (.5,2.65) .. (.5,3.1)
	-- (.5,3.25);
\draw (.5,.9) -- (2,.9);								
\draw (.5,.5) -- (2,.5);								
\draw[color=gray,dashed] (1.5,1.5) -- (2,1.5);		
\draw[color=gray] (.5,1.25) -- (1.5,1.25);			
\draw[color=white, line width=.7em] (.2,1.9) -- (1,1.9);
\draw[color=gray] (0,1.9) -- (1.5,1.9);
\draw (2,0) -- (2,2.15)
	.. controls (2,2.65) and (.5,2.65) .. (.5,3.1)
	-- (.5,3.25);
\bra{1.5}{3.1};
\bra{2}{3}
\ket{2.5}{.5};
\ket{2}{-.1}
\node at (-.23,1.9) {$W$};
\node at (.3,1.25) {$U$};
\node at (2.28,1.5) {$V^*$};
\node at (.3,.9) {$Z^*$};
\node at (.3,.5) {$F$};
\node at (0,-.5) {$H$};
\node at (.5,-.5) {$\chL$};
\node at (2,-.5) {$\chL$};
\node at (2.5,.1) {$H$};
\end{tikzpicture}} 
\right]
= 
\left[
\scalebox{\scalefactor}{
\begin{tikzpicture}[baseline=-\the\dimexpr\fontdimen22\textfont2\relax+1.3cm, line width=.1em]
\draw[color=gray] (0,-.25) -- (0,1.25)				
	.. controls (0,1.75) and (-.5,2) .. (-.5,2.5)
	.. controls (-.5,2.75) and (1,2.75) .. (1,3);
\draw (.5,-.25) -- (.5,.65)							
	.. controls (.5,1) and (1.75,1) .. (1.8,1.4)
	-- (1.8,1.5)
	.. controls (1.9,1.75) and (2.5,1.75) .. (2.5,2);
\draw[color=white, line width=.7em] (2.5,1) -- (2.5,1.45) 		
	.. controls (2.5,1.85) and (0,1.75) .. (0,2.25)
	-- (0,3.25);
\draw[color=gray] (2.5,1) -- (2.5,1.45) 				
	.. controls (2.5,1.85) and (0,1.75) .. (0,2.25)
	-- (0,3.25);
\draw[color=white,line width=.7em] (2,0) -- (2,.7)	
	.. controls (2,1.3) and (.5,1.3) .. (.5,1.8)
	-- (.4,3.25);
\draw (.5,.65) -- (2,.65);								
\draw (.5,.25) -- (2,.25);								
\draw[color=gray,dashed] (0,2.25) -- (.5,2.25);				
\draw[color=gray] (1.8,1.45) -- (2.5,1.45);				
\draw[color=gray] (-.5,2.5) -- (0,2.5);					
\draw (2,0) -- (2,.7)
	.. controls (2,1.3) and (.5,1.3) .. (.5,1.8)
	-- (.5,3.25);
\bra{1}{3.1};
\bra{2.5}{2.1}
\ket{2.5}{1.1};
\ket{2}{-.1}
\node at (-.6,2.7) {$W$};
\node at (2.75,1.45) {$U$};
\node at (.8,2.25) {$V^*$};
\node at (.3,.65) {$Z^*$};
\node at (.3,.25) {$F$};
\node at (0,-.5) {$H$};
\node at (.5,-.5) {$\chL$};
\node at (2,-.5) {$\chL$};
\node at (2.5,.7) {$H$};
\end{tikzpicture}} 
\right]
\nonumber\\
&\stackrel[\text{of }U]{\text{reg.}}= 
\left[
\scalebox{\scalefactor}{
\begin{tikzpicture}[baseline=-\the\dimexpr\fontdimen22\textfont2\relax+1.3cm, line width=.1em]
\draw[color=gray] (0,-.25) -- (0,1.25)				
	.. controls (0,1.75) and (-.5,2) .. (-.5,2.5)
	.. controls (-.5,2.75) and (1,2.75) .. (1,3);
\draw (.5,-.25) -- (.5,.65)							
	.. controls (.5,1) and (1.75,1) .. (1.8,1.4)
	-- (1.8,1.5);
\draw[color=white, line width=.7em] (1,1.75) 			
	.. controls (1,2) and (0,2) .. (0,2.25)
	-- (0,3.25);
\draw[color=gray] (1,1.75) 							
	.. controls (1,2) and (0,2) .. (0,2.25)
	-- (0,3.25);
\draw[color=white,line width=.7em] (2,0) -- (2,.65)	
	.. controls (2,1.225) and (.5,1.225) .. (.5,1.65)
	-- (.5,3.25);
\draw (.5,.65) -- (2,.65);								
\draw (.5,.25) -- (2,.25);								
\draw[color=gray,dashed] (0,2.25) -- (.5,2.25);			
\draw[color=gray] (-.5,2.5) -- (0,2.5);					
\draw (2,0) -- (2,.65)
	.. controls (2,1.225) and (.5,1.225) .. (.5,1.65)
	-- (.5,3.25);
\bra{1}{3.1};
\bra{1.8}{1.6}
\ket{1}{1.7};
\ket{2}{-.1}
\node at (-.35,2.85) {$W$};
\node at (.8,2.25) {$V^*$};
\node at (.3,.65) {$Z^*$};
\node at (.3,.25) {$F$};
\node at (0,-.5) {$H$};
\node at (.5,-.5) {$\chL$};
\node at (2,-.5) {$\chL$};
\end{tikzpicture}} 
\right]
\stackrel[\text{of }F]{\text{reg.}}= 
\left[
\scalebox{\scalefactor}{
\begin{tikzpicture}[baseline=-\the\dimexpr\fontdimen22\textfont2\relax+1.5cm, line width=.1em]
\draw[color=gray] (0,.25) -- (0,1.25)				
	.. controls (0,1.75) and (-.5,2) .. (-.5,2.5)
	.. controls (-.5,2.75) and (1,2.75) .. (1,3);
\draw (.5,.25) -- (.5,.75);
\draw[color=white, line width=.7em] (1,1.75) 			
	.. controls (1,2) and (0,2) .. (0,2.25)
	-- (0,3.25);
\draw[color=gray] (1,1.75) 							
	.. controls (1,2) and (0,2) .. (0,2.25)
	-- (0,3.25);
\draw[color=white,line width=.7em] (.5,1.25) -- (.5,3.25);	
\draw[color=gray,dashed] (0,2.25) -- (.5,2.25);			
\draw[color=gray] (-.5,2.5) -- (0,2.5);					
\draw (.5,1.25) -- (.5,3.25);
\bra{1}{3.1};
\ket{1}{1.7};
\ket{.5}{1.25};
\bra{.5}{.75};
\node at (-.35,2.85) {$W$};
\node at (.8,2.25) {$V^*$};
\node at (0,0) {$H$};
\node at (.5,0) {$\chL$};
\end{tikzpicture}} 
\right]
\stackrel{\eqref{eq:remove-unitaries}}= 
\left[
\scalebox{\scalefactor}{
\begin{tikzpicture}[baseline=-\the\dimexpr\fontdimen22\textfont2\relax+2cm, line width=.1em]
\draw[color=gray] (0,1.25) -- (0,1.25)				
	.. controls (0,1.75) and (-.5,2) .. (-.5,2.5)
	.. controls (-.5,2.75) and (1,2.75) .. (1,3);
\draw (.5,1.25) -- (.5,1.75);
\draw[color=white, line width=.7em] (0,2.25) -- (0,3.25);		
\draw[color=gray] (0,2.25) -- (0,3.25);					
\draw[color=white,line width=.7em] (.5,2.25) -- (.5,3.25);	
\draw[color=gray] (-.5,2.5) -- (0,2.5);					
\draw (.5,2.25) -- (.5,3.25);
\bra{1}{3.1};
\ket{0}{2.15};
\ket{.5}{2.25};
\bra{.5}{1.75};
\node at (-.35,2.85) {$W$};
\node at (0,1) {$H$};
\node at (.5,1) {$\chL$};
\end{tikzpicture}} 
\right]
\stackrel[\text{braiding}]{\text{reg. of}}= 
\left[
\scalebox{\scalefactor}{
\begin{tikzpicture}[baseline=-\the\dimexpr\fontdimen22\textfont2\relax+2cm, line width=.1em]
\draw[color=gray] (0,1.25) -- (0,1.25)				
	.. controls (0,1.75) and (-.5,2) .. (-.5,2.5)
	.. controls (-.5,2.75) and (.5,2.75) .. (.5,2.9);
\draw (1,1.25) -- (1,1.75);
\draw[color=white, line width=.7em] (0,2.25) -- (0,3.25);		
\draw[color=gray] (0,2.25) -- (0,3.25);				
\draw (1,2.25) -- (1,3.25);							
\draw[color=gray] (-.5,2.5) -- (0,2.5);					
\bra{.5}{3};
\ket{0}{2.15};
\ket{1}{2.25};
\bra{1}{1.75};
\node at (-.35,2.85) {$W$};
\node at (0,1) {$H$};
\node at (1,1) {$\chL$};
\end{tikzpicture}} 
\right]
\nonumber\\
&\stackrel[\text{of }W]{\text{reg.}}= 
\left[
\scalebox{\scalefactor}{
\begin{tikzpicture}[baseline=-\the\dimexpr\fontdimen22\textfont2\relax+.5cm, line width=.1em]
\draw[color=gray] (0,0) -- (0,.5);
\draw[color=gray] (0,1) -- (0,1.5);
\draw (.5,0) -- (.5,.5);
\draw (.5,1) -- (.5,1.5);
\ket{0}{1};
\bra{0}{.5};
\ket{.5}{1};
\bra{.5}{.5};
\node at (0,-.25) {$H$};
\node at (.5,-.25) {$\chL$};
\end{tikzpicture}} 
\right]
= \Kp(\H \ot \L) \ ,
\end{align}
\endgroup
which shows that $W \ltimes F$ is regular.

\medskip

\emph{Part 2:} 
Let $e$ be a fixed unit vector for $W$, and 
$\omega_{e,e}$ be the corresponding state 
$x \mapsto \langle e, xe \rangle$ on $\B(\K)$. 
Then we have, by regularity of $W \ltimes F$ and $W$, 
\begin{align}
&\Kp(\L) = (\omega_{e,e} \ot \id)(\Kp(\K \ot \L)) 
\nonumber\\
&= 
\left[
\scalebox{\scalefactor}{
\begin{tikzpicture}[baseline=-\the\dimexpr\fontdimen22\textfont2\relax+1.3cm, line width=.1em]
\draw[color=gray] (0,0) -- (0,2.15)				
	.. controls (0,2.6) and (1.5,2.6) .. (1.5,3);
\draw (.5,-.25) -- (.5,2.1)							
	.. controls (.5,2.5) and (2,2.5) .. (2,2.9);
\draw[color=white, line width=.7em] (1.5,2)		
	.. controls (1.5,2.4) and (0,2.4) .. (0,3);
\draw[color=gray] (1.5,0) -- (1.5,.25) 				
	.. controls (1.5,.5) and (2.4,.5) .. (2.4,.75)
	.. controls (2.4,1) and (1.5,1) .. (1.5,1.25)
	-- (1.5,2)
	.. controls (1.5,2.4) and (0,2.4) .. (0,3);
\draw[color=white,line width=.7em] (2,0) -- (2,2.15)		
	.. controls (2,2.65) and (.5,2.65) .. (.5,3.1)
	-- (.5,3.25);
\draw (.5,.75) -- (2,.75);							
\draw[color=gray] (1.5,.25) -- (2,.25);
\draw[color=gray,dashed] (1.5,1.25) -- (2,1.25);
\draw[color=gray] (.5,1.5) -- (1.5,1.5);			
\draw[color=white, line width=.7em] (.2,1.9) -- (1,1.9);
\draw[color=gray] (0,1.9) -- (1.5,1.9);
\draw (2,0) -- (2,2.15)
	.. controls (2,2.65) and (.5,2.65) .. (.5,3.1)
	-- (.5,3.25);
\bra{1.5}{3.1};
\bra{2}{3}
\ket{1.5}{-.1};
\ket{2}{-.1}
\node[draw,fill=white,inner sep=.15em] at (0,0) {$e$};
\node[draw,fill=white,inner sep=.1em] at (0,3) {$e^*$};
\node at (-.23,1.9) {$W$};
\node at (.3,1.5) {$U$};
\node at (2.28,1.25) {$V^*$};
\node at (.3,.75) {$F$};
\node at (2.2,.25) {$V$};
\node at (.5,-.5) {$\chL$};
\node at (1.5,-.5) {$H$};
\node at (2,-.5) {$\chL$};
\end{tikzpicture}} 
\right]
\stackrel{\eqref{eq:prop-regularity-of-semidirect-product--1}}= 
\left[
\scalebox{\scalefactor}{
\begin{tikzpicture}[baseline=-\the\dimexpr\fontdimen22\textfont2\relax+1.3cm, line width=.1em]
\draw[color=gray] (0,1.25)						
	.. controls (0,1.75) and (-.5,2) .. (-.5,2.5)
	.. controls (-.5,2.75) and (1,2.75) .. (1,3);
\draw (.5,-.25) -- (.5,.65)							
	.. controls (.5,1) and (1.75,1) .. (1.8,1.4)
	-- (1.8,1.5);
\draw[color=white, line width=.7em] (1,1.75) 			
	.. controls (1,2) and (0,2) .. (0,2.25)
	-- (0,3.25);
\draw[color=gray] (1,1.75) 							
	.. controls (1,2) and (0,2) .. (0,2.25)
	-- (0,3.25);
\draw[color=white,line width=.7em] (2,0) -- (2,.65)	
	.. controls (2,1.225) and (.5,1.225) .. (.5,1.65)
	-- (.5,3.25);
\draw (.5,.65) -- (2,.65);								
\draw (.5,.25) -- (2,.25);								
\draw[color=gray,dashed] (0,2.25) -- (.5,2.25);			
\draw[color=gray] (-.5,2.5) -- (0,2.5);					
\draw (2,0) -- (2,.65)
	.. controls (2,1.225) and (.5,1.225) .. (.5,1.65)
	-- (.5,3.25);
\bra{1}{3.1};
\bra{1.8}{1.6}
\ket{1}{1.7};
\ket{2}{-.1}
\node[draw,fill=white,inner sep=.15em] at (0,1.25) {$e$};
\node[draw,fill=white,inner sep=.1em] at (0,3.1) {$e^*$};
\node at (-.5,2.8) {$W$};
\node at (.8,2.25) {$V^*$};
\node at (.3,.65) {$Z^*$};
\node at (.3,.25) {$F$};
\node at (.5,-.5) {$\chL$};
\node at (2,-.5) {$\chL$};
\end{tikzpicture}} 
\right]
\stackrel[\ref{df:fixed-vectors}]{\text{Def.}}= 
\left[
\scalebox{\scalefactor}{
\begin{tikzpicture}[baseline=-\the\dimexpr\fontdimen22\textfont2\relax+1.3cm, line width=.1em]
\draw[color=gray] (-.5,2.25) -- (-.5,2.5)				
	.. controls (-.5,2.75) and (1,2.75) .. (1,3);
\draw (.5,-.25) -- (.5,.65)							
	.. controls (.5,1) and (1.75,1) .. (1.8,1.4)
	-- (1.8,1.5);
\draw[color=white, line width=.7em] (1,1.75) 			
	.. controls (1,2) and (0,2) .. (0,2.25)
	-- (0,3.25);
\draw[color=gray] (1,1.75) 							
	.. controls (1,2) and (0,2) .. (0,2.25)
	-- (0,3.25);
\draw[color=white,line width=.7em] (2,0) -- (2,.65)	
	.. controls (2,1.225) and (.5,1.225) .. (.5,1.65)
	-- (.5,3.25);
\draw (.5,.65) -- (2,.65);								
\draw (.5,.25) -- (2,.25);								
\draw[color=gray,dashed] (0,2.25) -- (.5,2.25);			
\draw (2,0) -- (2,.65)
	.. controls (2,1.225) and (.5,1.225) .. (.5,1.65)
	-- (.5,3.25);
\bra{1}{3.1};
\bra{1.8}{1.6}
\ket{1}{1.7};
\ket{2}{-.1}
\node[draw,fill=white,inner sep=.15em] at (-.5,2.25) {$e$};
\node[draw,fill=white,inner sep=.1em] at (0,3.1) {$e^*$};
\node at (.8,2.25) {$V^*$};
\node at (.3,.65) {$Z^*$};
\node at (.3,.25) {$F$};
\node at (.5,-.5) {$\chL$};
\node at (2,-.5) {$\chL$};
\end{tikzpicture}} 
\right]
\stackrel[\ref{rmk:fixed-vectors}]{\text{Rmk.}}= 
\left[
\scalebox{\scalefactor}{
\begin{tikzpicture}[baseline=-\the\dimexpr\fontdimen22\textfont2\relax+1.3cm, line width=.1em]
\draw (.5,-.25) -- (.5,.65)							
	.. controls (.5,1) and (1.75,1) .. (1.8,1.4)
	-- (1.8,1.5);
\draw[color=white, line width=.7em] (1,1.75) 			
	.. controls (1,2) and (0,2) .. (0,2.25)
	-- (0,3.25);
\draw[color=gray] (1,1.75) 							
	.. controls (1,2) and (0,2) .. (0,2.25)
	-- (0,3.25);
\draw[color=white,line width=.7em] (2,0) -- (2,.65)	
	.. controls (2,1.225) and (.5,1.225) .. (.5,1.65)
	-- (.5,3.25);
\draw (.5,.65) -- (2,.65);								
\draw (.5,.25) -- (2,.25);								
\draw (2,0) -- (2,.65)
	.. controls (2,1.225) and (.5,1.225) .. (.5,1.65)
	-- (.5,3.25);
\bra{1.8}{1.6}
\ket{1}{1.7};
\ket{2}{-.1}
\node[draw,fill=white,inner sep=.1em] at (0,3.1) {$e^*$};
\node at (.3,.65) {$Z^*$};
\node at (.3,.25) {$F$};
\node at (.5,-.5) {$\chL$};
\node at (2,-.5) {$\chL$};
\end{tikzpicture}} 
\right]
\nonumber\\
&= C(F) \ .
\end{align}
\end{proof}
\section{Bi-regularity}
\label{sec:biregularity}

To motivate this section, let us consider a 
multiplicative unitary $W \in \U(\H \ot \H)$ in the category 
$(\Hilb,\ot,\C,\Sigma)$ of Hilbert spaces equipped 
with the symmetric braiding given by the tensor flip $\Sigma$.
If $W$ is \emph{manageable}, see e.g. 
\cite{Woronowicz:Mult_unit_to_Qgrp}*{Def.\,1.2} for a definition, 
then the algebras $\hat A(W)$ and $\hat A(\hat W)$ 
are $C^*$-algebras and come equipped with anti-automorphisms 
$\hat R$ and $R$ satisfying 
$(\hat R \ot R)(W) = W$; see \cite{Soltan-Woronowicz:Multiplicative_unitaries}*{Lem.\,40}.
If $W$ is \emph{manageable} and \emph{regular}, then 
we may apply $\hat R \ot R$ to 
\begin{align}
[(\hatA(W) \ot \id)W(\id \ot \hatA(\hat W))] 
	= \hatA(W) \ot_\mathrm{min} \hatA(\hat W) 
\end{align}
from Theorem \ref{thm:F-as-multiplier}.2 
to obtain 
\begin{align}\label{eq:opp-regularity}
[(\id \ot \hatA(\hat W))W(\hatA(W) \ot \id)] 
	= \hatA(W) \ot_\mathrm{min} \hatA(\hat W) \ .
\end{align}
As both $\hatA(W)$ and $\hatA(\hat W)$ act non-degenerately 
on $\H$, hence, on $\Kp(\H)$, multiplying 
\eqref{eq:opp-regularity} with 
$\id \ot \Kp(\H)$ from the left, and with 
$\Kp(\H) \ot \id$ from the right, 
we obtain
\begin{align}
[(\id \ot \Kp(\H))W(\Kp(\H) \ot \id)] 
	= \Kp(\H \ot \H) \ .
\end{align}
From this we see that $W$ also satisfies the following 
regularity property:
\begin{align}
\Kp(\H) = [(\omega \ot \id)(\Sigma W) \,|\, \omega \in \B(\H)_*] \ .
\end{align}
A regular multiplicative unitary with this 
additional property is called \emph{bi-regular} 
in \cite{Baaj-Skandalis:Unitaires}*{Def.\,3.10}.
We have arrived at

\begin{prop}
Let $W$ be a multiplicative unitary in $(\Hilb,\ot,\C,\Sigma)$. 
If $W$ is regular and manageable, then it 
is bi-regular.
\end{prop}

Since manageable multiplicative unitaries form an 
important class, we are now going to generalise the 
notion of bi-regularity to braided multiplicative 
unitaries.


\begin{df}\label{df:bi-semi-regularity}
Let $(\Cat,\ot,\one,c)$ be a unitarily braided 
concrete monoidal $W^*$-category. 
\begin{enumerate}
\item 
The braiding $c$ on $\Cat$ is said to be 
\emph{bi-\textnormal{(}semi-\textnormal{)}regular}, if it is (semi-)regular on $(\Cat,\ot,\one)$ 
and on $(\Cat,\ot^\mathrm{op},\one)$.
\item 
\sloppy
A multiplicative unitary $F \in \U_{\Cat}(L \ot L)$ 
is called \emph{bi-\textnormal{(}semi-\textnormal{)}regular}, if $F$ is \mbox{(semi-)}regular 
in $(\Cat,\ot,\one,c)$ and $F^*$ is (semi-)regular in 
$(\Cat,\ot^\mathrm{op},\one,c^{-1})$.
\end{enumerate}
\end{df}

Let us first comment on part 1 of the above definition.
Spelled out explicitly, a braiding $c$ on $\Cat$ is bi-regular, if 
\begin{align}
[(\id \ot \omega)(c_{H,K})\,|\, \omega \in \B(\K,\H)_*]
= \Kp(\H,\K) 
= [(\omega \ot \id)(c_{H,K})\,|\, \omega \in \B(\H,\K)_*]
\end{align}
holds for all $H,K \in \Cat$. 
Braidings on representation categories of compact 
quantum groups are bi-regular automatically. 
Before proving this, we introduce 
\emph{$^*$strongly rigid} categories.
Recall that the \emph{$^*$strong topology} on $\Hom(X,Y)$ is 
the locally convex topology generated by the semi-norms 
$a \mapsto \sqrt{\varphi(a^*a)}$ and 
$a \mapsto \sqrt{\varphi(aa^*)}$, where $\varphi$ runs over 
all positive elements of $\Hom(X,X)_*$. 

\begin{df}
Let $(\Cat,\ot,\one)$ be a monoidal $W^*$-category and $H \in \Cat$ an object. 
We say that an object $\bar H \in \Cat$ is a $^*$strong left-dual for 
$H$, if the sets 
\begin{align*}
\left\{ \left.(x \ot \id_{H})\circ(\id_{H} \ot y) 
	\,\right|\, x \in \Hom_\Cat(H \ot \bar H,\one),\, y \in \Hom_\Cat(\one,\bar H \ot H) \right\} 
	\subseteq \End_\Cat(H) \ , \nonumber\\
\left\{ \left.(\id_{\bar H} \ot x)\circ(y \ot \id_{\bar H}) 
	\,\right|\, x \in \Hom_\Cat(H \ot \bar H,\one),\, y \in \Hom_\Cat(\one,\bar H \ot H) \right\} 
	\subseteq \End_\Cat(\bar H) \ , 
\end{align*}
are $^*$strongly dense in $\End_\Cat(H)$, $\End_\Cat(\bar H)$. 
We call a monoidal $W^*$-category $(\Cat,\ot,\one)$ \emph{$^*$strongly rigid}, 
if every object $H \in \Cat$ admits a $^*$strong left-dual. 
\end{df}

The category $(\Hilb,\ot,\C)$ and, more generally, 
representation categories of compact quantum groups 
are $^*$strongly rigid.
We have:

\begin{prop}
Let $(\Cat,\ot,\one)$ be a $^*$strongly rigid concrete 
monoidal $W^*$-category, and let $c$ be a unitary 
braiding on $\Cat$. 
Then $c$ is bi-regular.
\end{prop}
\begin{proof}
As a normal fibre functor is also 
continuous with respect to the $^*$strong topology, 
we have 
in a $^*$strongly rigid and concrete $W^*$-category 
\begin{align}\label{eq:st-strong-rigidity--1}
\left[ 
\scalebox{\scalefactor}{
\begin{tikzpicture}[baseline=-\the\dimexpr\fontdimen22\textfont2\relax+.5cm, line width=.1em]
\draw (0,0) -- (0,1);
\bra{0}{1};
\node[anchor=west,inner sep=.2em] at (0,.15) {$H$};
\end{tikzpicture}} 
\right]
=
\left[ \left.
\scalebox{\scalefactor}{
\begin{tikzpicture}[baseline=-\the\dimexpr\fontdimen22\textfont2\relax+.75cm, line width=.1em]
\draw (0,0) -- (0,1.25) -- (1,.5) -- (1,1.5);
\bra{1}{1.5};
\node[draw,fill=white] at (.1,1.25) {$x$};
\node[draw,fill=white] at (.9,.5) {$y$};
\node[anchor=west,inner sep=.2em] at (0,.15) {$H$};
\end{tikzpicture}} 
\ \right|
{x \in \Hom_\Cat(H \ot \bar H, \one),\atop
y \in \Hom_\Cat(\one, \bar H \ot H)}
\right]
\ ,
\end{align}
as well as
\begin{align}\label{eq:st-strong-rigidity--2}
\left[ 
\scalebox{\scalefactor}{
\begin{tikzpicture}[baseline=-\the\dimexpr\fontdimen22\textfont2\relax+.5cm, line width=.1em]
\draw (0,0) -- (0,1);
\ket{0}{0};
\node[anchor=west,inner sep=.2em] at (0,.85) {$\bar H$};
\end{tikzpicture}} 
\right]
=
\left[ \left.
\scalebox{\scalefactor}{
\begin{tikzpicture}[baseline=-\the\dimexpr\fontdimen22\textfont2\relax+.75cm, line width=.1em]
\draw (-.3,1.5) -- (-.1,.25) -- (.1,.25) -- (.2,1);
\bra{.2}{1};
\node[draw,fill=white] at (0,.25) {$y$};
\node[anchor=west,inner sep=.2em] at (-.3,1.35) {$\bar H$};
\end{tikzpicture}} 
\ \right|
y \in \Hom_\Cat(\one, \bar H \ot H)
\right]
\end{align}
and
\begin{align}\label{eq:st-strong-rigidity--3}
\left[ 
\scalebox{\scalefactor}{
\begin{tikzpicture}[baseline=-\the\dimexpr\fontdimen22\textfont2\relax+.5cm, line width=.1em]
\draw (0,0) -- (0,1);
\bra{0}{1};
\node[anchor=west,inner sep=.2em] at (0,.15) {$H$};
\end{tikzpicture}} 
\right]
=
\left[ \left.
\scalebox{\scalefactor}{
\begin{tikzpicture}[baseline=-\the\dimexpr\fontdimen22\textfont2\relax+.75cm, line width=.1em]
\draw (-.3,0) -- (-.1,1.25) -- (.1,1.25) -- (.2,.5);
\ket{.2}{.5};
\node[draw,fill=white] at (0,1.25) {$x$};
\node[anchor=west,inner sep=.2em] at (-.3,.15) {$H$};
\end{tikzpicture}} 
\ \right|
x \in \Hom_\Cat(H \ot \bar H, \one)
\right]
\ .
\end{align}
Hence,
\begin{align}
\left[ 
\scalebox{\scalefactor}{
\begin{tikzpicture}[baseline=-\the\dimexpr\fontdimen22\textfont2\relax+.75cm, line width=.1em]
\draw (0,0) -- (0,.25);
\draw (0,1.25) -- (0,1.5);
\braid[border height=.0cm,height=1cm,width=1cm] at (0,1.25) s_1^{-1};
\bra{1}{1.25};
\ket{1}{.25};
\node[anchor=west,inner sep=.2em] at (0,1.35) {$K$};
\node[anchor=west,inner sep=.2em] at (0,.15) {$H$};
\end{tikzpicture}} 
\right]
&\stackrel{\eqref{eq:st-strong-rigidity--1}}=
\left[ \left.
\scalebox{\scalefactor}{
\begin{tikzpicture}[baseline=-\the\dimexpr\fontdimen22\textfont2\relax+.875cm, line width=.1em]
\draw (0,0) -- (0,.25);
\draw (0,1.25) -- (0,1.75);
\draw (1,1.25) -- (1,1.5) -- (2,.75) -- (2,1.5);
\braid[border height=.0cm,height=1cm,width=1cm] at (0,1.25) s_1^{-1};
\bra{2}{1.5};
\ket{1}{.25};
\node[draw,fill=white] at (1.1,1.5) {$x$};
\node[draw,fill=white] at (1.9,.75) {$y$};
\node[anchor=west,inner sep=.2em] at (0,1.6) {$K$};
\node[anchor=west,inner sep=.2em] at (0,.15) {$H$};
\end{tikzpicture}} 
\ \right|
{x \in \Hom_\Cat(H \ot \bar H, \one),\atop
y \in \Hom_\Cat(\one, \bar H \ot H)}
\right]
\nonumber\\
&\stackrel{\eqref{eq:st-strong-rigidity--2}}=
\left[ \left.
\scalebox{\scalefactor}{
\begin{tikzpicture}[baseline=-\the\dimexpr\fontdimen22\textfont2\relax+.875cm, line width=.1em]
\draw (.5,0) -- (.8,1.3);
\draw (2,1.25) -- (2,1.75);
\braid[border height=.0cm,height=1cm,width=1cm] at (1,1.25) s_1;
\ket{2}{.25};
\ket{1}{.25};
\node[draw,fill=white] at (.9,1.3) {$x$};
\node[anchor=east,inner sep=.2em] at (2,1.6) {$K$};
\node[anchor=east,inner sep=.2em] at (.5,.15) {$H$};
\end{tikzpicture}} 
\ \right|
x \in \Hom_\Cat(H \ot \bar H, \one)
\right]
\nonumber\\
&\stackrel{\eqref{eq:remove-unitaries}}=
\left[ \left.
\scalebox{\scalefactor}{
\begin{tikzpicture}[baseline=-\the\dimexpr\fontdimen22\textfont2\relax+.875cm, line width=.1em]
\draw (.5,0) -- (.8,1.3) -- (1,1.3) -- (1.2,.5);
\draw (2,.75) -- (2,1.75);
\ket{2}{.75};
\ket{1.2}{.5};
\node[draw,fill=white] at (.9,1.3) {$x$};
\node[anchor=east,inner sep=.2em] at (2,1.6) {$K$};
\node[anchor=west,inner sep=.2em] at (.5,.15) {$H$};
\end{tikzpicture}} 
\ \right|
x \in \Hom_\Cat(H \ot \bar H, \one)
\right]
\stackrel{\eqref{eq:st-strong-rigidity--3}}=
\left[ 
\scalebox{\scalefactor}{
\begin{tikzpicture}[baseline=-\the\dimexpr\fontdimen22\textfont2\relax, line width=.1em]
\draw (0,.25) -- (0,.75);
\draw (0,-.25) -- (0,-.75);
\ket{0}{.25};
\bra{0}{-.25};
\node at (.25,.6) {$K$};
\node at (.25,-.6) {$H$};
\end{tikzpicture}} 
\right] 
\ 
\end{align}
shows the regularity of the braiding $c$. 
Its bi-regularity is shown in a similar manner.
\end{proof}

We turn now to the second part of Definition 
\ref{df:bi-semi-regularity}.
A multiplicative unitary $F$ is bi-regular if 
\begin{align}\label{eq:biregularity-in-pictures}
\left[ 
\scalebox{\scalefactor}{
\begin{tikzpicture}[baseline=-\the\dimexpr\fontdimen22\textfont2\relax+.875cm, line width=.1em]
\draw (0,0) -- (0,.5);
\draw (0,1.5) -- (0,1.75);
\draw (1,.25) -- (1,.5);
\draw (0,.5) -- (1,.5);
\braid[border height=.0cm,height=1cm,width=1cm] at (0,1.5) s_1;
\bra{1}{1.5};
\ket{1}{.25};
\node at (.2,.15) {$L$};
\node at (-.2,.5) {$F$};
\end{tikzpicture}} 
\right]
\stackrel{\text{(*)}}=
\left[ 
\scalebox{\scalefactor}{
\begin{tikzpicture}[baseline=-\the\dimexpr\fontdimen22\textfont2\relax, line width=.1em]
\draw (0,.25) -- (0,1);
\draw (0,-.25) -- (0,-1);
\node[draw,fill=white,minimum width=.0em,
	inner sep=.2em,
	regular polygon,regular polygon sides=3,
	shape border rotate=60,
	yscale=.8,label=center:] at (0,.25) {};
\node[draw,fill=white,minimum width=.0em,
	inner sep=.2em,
	regular polygon,regular polygon sides=3,
	shape border rotate=0,
	yscale=.8,label=center:] at (0,-.25) {};
\node at (.25,.83) {$L$};
\node at (.25,-.83) {$L$};
\end{tikzpicture}} 
\right] 
\stackrel{\text{(**)}}=
\left[ 
\scalebox{\scalefactor}{
\begin{tikzpicture}[baseline=-\the\dimexpr\fontdimen22\textfont2\relax+.875cm, line width=.1em]
\draw (1,0) -- (1,.5);
\draw (1,1.5) -- (1,1.75);
\draw (0,.25) -- (0,.5);
\draw[dashed] (0,.5) -- (1,.5);
\braid[border height=.0cm,height=1cm,width=1cm] at (0,1.5) s_1^{-1};
\bra{0}{1.5};
\ket{0}{.25};
\node[anchor=east] at (1,.15) {$L$};
\node[anchor=west] at (1,.5) {$F^*$};
\end{tikzpicture}} 
\right]
=
\left[ 
\scalebox{\scalefactor}{
\begin{tikzpicture}[baseline=-\the\dimexpr\fontdimen22\textfont2\relax+.875cm, line width=.1em]
\draw (1,0) -- (1,.25);
\draw (1,1.25) -- (1,1.75);
\draw (0,1.25) -- (0,1.5);
\draw (0,1.25) -- (1,1.25);
\braid[border height=.0cm,height=1cm,width=1cm] at (0,1.25) s_1;
\bra{0}{1.5};
\ket{0}{.25};
\node[anchor=east] at (1,.15) {$L$};
\node[anchor=west] at (1,1.25) {$F$};
\end{tikzpicture}} 
\right]^* \ .
\end{align}
Note that the right-hand side picture in 
\eqref{eq:biregularity-in-pictures} is obtained 
from the one on the left-hand side by a 
rotation about 180$^\circ$. 
Therefore, if we have a pictorial proof of 
some statement which makes use of regularity, i.e.\@ (*) in 
\eqref{eq:biregularity-in-pictures}, 
then its rotation about 180$^\circ$ 
gives a statement with regularity replaced 
by the condition (**) in \eqref{eq:biregularity-in-pictures}.
In combination we obtain a statement involving 
bi-regularity.
We shall use this observation in the following 
to obtain analogues of results for regular 
multiplicative unitaries in the bi-regular 
setting.

Given a concrete monoidal $W^*$-category 
$\Cat$ with bi-semi-regular braiding $c$ 
and a multiplicative unitary 
$F \in \U_{\Cat}(L \ot L)$, we define -- 
in addition to the algebra $\hatA(F)$ -- 
\begin{align}
A(F) := [(\omega \ot \id)(F)\,|\, \omega \in \B(\L)_*] \ . 
\end{align}
Then, in analogy to Proposition \ref{prop:F-slices-are-non-deg-algebras}, 
we have:

\begin{prop}
\begin{enumerate}
\item $A(F) \subseteq \B(\L)$ is a subalgebra.
\item $[A(F)\L] = \L$, i.e.\@ $A(F)$ acts non-degenerately. 
\end{enumerate}
\end{prop}

The proof is performed in a similar way as the proof 
of Proposition \ref{prop:F-slices-are-non-deg-algebras}, 
but with pictures rotated by 180$^\circ$. 
In this manner we also obtain -- cf.\@ 
Proposition \ref{prop:hat-AF-is-a-C*-alg} -- 

\begin{prop}
Let $\Cat$ be a bi-semi-regularly braided concrete monoidal $W^*$-category, 
and let $L \in \Cat$ be an object. 
If $F \in \U_{\Cat}(L \ot L)$ is a bi-semi-regular 
multiplicative unitary, 
then $A(F)$ is a $C^*$-algebra. 
\end{prop}

Let $\Cat$ be a bi-regularly braided concrete monoidal $W^*$-category. 
We define a category $A(\Cat)$ in a similar way as $\hatA(\Cat)$ 
in Section \ref{ssec:category-hat-A(C)}: 
Objects of $A(\Cat)$ are pairs $(H,A)$, where $H \in \Cat$ 
and $A \subseteq \B(\H)$ is a $C^*$-subalgeba 
of the form $A = A(X) := [(\omega \ot \id)(X)\,|\, \omega \in \B(\K)_*]$, 
for some $K \in \Cat$ and morphism $X \in \End_\Cat(K \ot H)$. 
Morphisms $(H_1,A_1) \to (H_2,A_2)$ 
in $A(\Cat)$ are
non-degenerate $^*$-homomorphisms $f:\, A_1 \to \Malg(A_2)$ 
of the form 
$A_1 \ni a \mapsto V(a \ot \id_K)V^*$, 
for some $K \in \Cat$ and some partial isometry 
$V \in \Hom_\Cat(H_1 \ot K, H_2)$. 
With 
$(H_1,A_1) \bt (H_2,A_2) := (H_1 \ot H_2, A_1 \bt A_2)$, 
for objects $(H_1,A_1),(H_2,A_2) \in A(\Cat)$, 
where 
\begin{align}\label{eq:bt-on-algebras}
A_1 \bt A_2 := \bigl[(A_1 \ot \id_{\H_2}) \cdot 
	c_{H_1,H_2}^{-1}(A_2 \ot \id_{\H_1})c_{H_1,H_2}\bigr] \subseteq \B(\H_1 \ot \H_2) \ ,
\end{align}
we get a monoidal category of $C^*$-algebras 
$(A(\Cat),\bt)$.

Furthermore, we obtain:

\begin{thm}
Let $\Cat$ be a bi-regularly braided concrete monoidal $W^*$-category, 
let $L \in \Cat$ be an object with underlying 
Hilbert space $\L$, and let $F \in \U_{\Cat}(L \ot L)$ 
be a bi-regular multiplicative unitary. 
Then we have 
\begin{enumerate}
\item $F \in \Malg(A(\hat F) \bt A(F))$ 
\item $[(A(\hat F) \bt \id_\L)F(\id_\L \bt A(F))] = A(\hat F) \bt A(F)$. 
\end{enumerate}
\end{thm}

\begin{prop}
Let $\Cat$ be a bi-regularly braided concrete monoidal $W^*$-category, 
let $L \in \Cat$ be an object with underlying 
Hilbert space $\L$, and let $F \in \U_{\Cat}(L \ot L)$ 
be a bi-regular multiplicative unitary. 
For $a \in A(F)$, put $\Delta(a) := F(a \ot \id_L)F^*$. 
Then $(A,\Delta)$ is a bi-simplifiable $C^*$-bialgebra in 
$A(\Cat)$.
\end{prop}

Unfortunately, we are unable to prove an analogue of 
Proposition \ref{prop:YD-cats-of-good-regular-MUs-have-regular-braiding}. 
The rotate-by-180$^\circ$ trick does not work 
since the graphical representation 
of the braiding $\Phi$ on 
the Yetter-Drinfeld category $\FYDF$ 
is not symmetric under rotations by 180$^\circ$.
\section{Acknowledgement}
The first author is fully, and the second author is partially supported by the German Research 
Foundation (Deutsche Forschungsgemeinschaft (DFG)) through the Research 
Training Group 1493. The second author is also partially supported by the Fields--Ontario 
postdoctoral fellowship. We thank Ralf Meyer and Folkert M\"uller-Hoissen for helpful comments 
and suggestions, and the Fields Institute for its hospitality.

\end{document}